\documentclass[reqno]{amsart}
\usepackage{color}
\usepackage{amsthm}
\usepackage{amsmath}
\usepackage{epsfig}
\usepackage{graphicx}
\usepackage{amssymb}
\usepackage{latexsym}
\usepackage{pdfpages}

\usepackage{enumitem}

\usepackage{ulem} %para poder tachar: \sout{loquequierotachar}

\usepackage{ifpdf}
\newcommand{\res}{\!\!\mathop{\hbox{
                                \vrule height 7pt width .5pt depth 0pt
                                \vrule height .5pt width 6pt depth 0pt}}
                                \nolimits}

\def\z{{\bf z}}

\ifpdf %if using pdfLaTeX in PDF mode
  \usepackage[hidelinks]{hyperref}
\else %if using LaTeX or pdfLaTeX in DVI mode
\fi %%% fin del ifpdf

 \usepackage[usenames,dvipsnames]{pstricks}
 \usepackage{epsfig}
 \usepackage{pst-grad} % For gradients
 \usepackage{pst-plot} % For axes

\usepackage{color}

\definecolor{olive(ryb)}{rgb}{0.42, 0.40, 0.18}
\newtheorem{theorem}{Theorem}[section]
\newtheorem{lemma}[theorem]{Lemma}
\newtheorem{definition}[theorem]{Definition}
\newtheorem{proposition}[theorem]{Proposition}
\newtheorem{corollary}[theorem]{Corollary}
\newtheorem{remark}[theorem]{Remark}
\newtheorem{example}[theorem]{Example}
\newtheorem*{theorem*}{\it Theorem}

\def\vint_#1{\mathchoice%
          {\mathop{\kern 0.2em\vrule width 0.6em height 0.69678ex depth -0.58065ex
                  \kern -0.8em \intop}\nolimits_{\kern -0.4em#1}}%
          {\mathop{\kern 0.1em\vrule width 0.5em height 0.69678ex depth -0.60387ex
                  \kern -0.6em \intop}\nolimits_{#1}}%
          {\mathop{\kern 0.1em\vrule width 0.5em height 0.69678ex
              depth -0.60387ex
                  \kern -0.6em \intop}\nolimits_{#1}}%
          {\mathop{\kern 0.1em\vrule width 0.5em height 0.69678ex depth -0.60387ex
                  \kern -0.6em \intop}\nolimits_{#1}}}
\def\vintslides_#1{\mathchoice%
          {\mathop{\kern 0.1em\vrule width 0.5em height 0.697ex depth -0.581ex
                  \kern -0.6em \intop}\nolimits_{\kern -0.4em#1}}%
          {\mathop{\kern 0.1em\vrule width 0.3em height 0.697ex depth -0.604ex
                  \kern -0.4em \intop}\nolimits_{#1}}%
          {\mathop{\kern 0.1em\vrule width 0.3em height 0.697ex depth -0.604ex
                  \kern -0.4em \intop}\nolimits_{#1}}%
          {\mathop{\kern 0.1em\vrule width 0.3em height 0.697ex depth -0.604ex
                  \kern -0.4em \intop}\nolimits_{#1}}}

\def\R{\mathbb R}
\def\N{\mathbb N}

\numberwithin{equation}{section}

\def\1{\raisebox{2pt}{\rm{$\chi$}}}

\definecolor{violet(ryb)}{rgb}{0.53, 0.0, 0.69}
\definecolor{internationalorange}{rgb}{1.0, 0.31, 0.0}
\usepackage{mathtools}
\mathtoolsset{showonlyrefs}
\begin{document}

\title[The Heat Flow  on Metric Random Walk Spaces]{\bf The Heat Flow  on Metric Random Walk Spaces}

\author[J. M. Maz\'on, M. Solera, J. Toledo]{Jos\'e M. Maz\'on, Marcos Solera and   Juli\'{a}n Toledo}

 \address{ J. M. Maz\'{o}n: Departamento de An\'{a}lisis Matem\'{a}tico,
Univ. Valencia, Dr. Moliner 50, 46100 Burjassot, Spain.
 {\tt mazon@uv.es}}
\address{ M. Solera: Departamento de An\'{a}lisis Matem\'{a}tico,
Univ. Valencia, Dr. Moliner 50, 46100 Burjassot, Spain.
 {\tt  marcos.solera@uv.es }}
\address{J. Toledo: Departamento de An\'{a}lisis Matem\'{a}tico,
Univ. Valencia, Dr. Moliner 50, 46100 Burjassot, Spain.
 {\tt toledojj@uv.es }}

%\hfill\break\indent

\keywords{Random walk,  nonlocal operators, Logarithmic-Sobolev inequalities, Cheeger inequality, Ollivier-Ricci curvature, Bakry-\'{E}mery curvature-dimension condition, Concentration of measures, Transport inequalities\\
\indent 2010 {\it Mathematics Subject Classification:} 35K05, 47D07, 05C81, 31C20, 26D10, 45C99.
}

\setcounter{tocdepth}{1}

\date{\today}

\begin{abstract}
In this paper we study the Heat Flow on Metric Random Walk Spaces, which unifies into a broad
framework the heat flow on locally finite weighted connected graphs, the heat flow determined by finite Markov chains and some nonlocal evolution problems. We give different characterizations of the ergodicity and prove that a metric random walk space with positive  Ollivier-Ricci curvature is ergodic. Furthermore, we prove a Cheeger inequality and, as a consequence,  we show that a Poincar\'{e} inequality holds if, and only if,  an isoperimetric inequality holds.  We also study the Bakry-\'{E}mery curvature-dimension condition and its relation with functional inequalities like the Poincar\'{e} inequality  and the transport-information inequalities.

\end{abstract}

\maketitle

{ \renewcommand\contentsname{Contents}
\setcounter{tocdepth}{3}
\addtolength{\parskip}{-0.2cm}
{\small \tableofcontents}
\addtolength{\parskip}{0.2cm} }

\section{Introduction and Preliminaries}

A  metric random walk space is a metric space $(X,d)$ together with a family  $m = (m_x)_{x \in X}$ of probability measures that  encode the jumps of a Markov chain. Given an initial mass distribution $\mu$ on $X$,  the measure $\mu \ast m$ given by
$$\mu \ast m (A):= \int_X m_x(A) d\mu(x), \quad \hbox{for all Borel sets} \ A \subset X,$$
describes the new mass distribution after a jump.  Associated with $m$, the {\it Laplace operator} $\Delta_m$ is defined as
$$\Delta_m f (x):= \int_X (f(y) - f(x)) dm_x(y).$$
Assuming that there exists an invariant and reversible measure $\nu$ for the random walk,  the operator $- \Delta_m$ generates in $L^2(X, \nu)$ a Markovian semigroup $(e^{t \Delta_m})_{t \geq 0}$ (Theorem~\ref{generator}) called the {\it heat flow} on the metric random walk space, which unifies into a broad
framework the heat flow on graphs, the heat flow determined by finite Markov chains and also some nonlocal heat flows.

It is of great importance in many applications to understand the behaviour of the semigroup $(e^{t \Delta_m})_{t \geq 0}$ as $t \to \infty$.  In this regard, we introduce a new concept, called random walk connectedness or $m$-connectedness of the metric random walk space, which is related to the geometry  of the metric random walk space. We then prove that it is equivalent to the infinite speed of propagation of the heat flow (Theorem~\ref{despf01}) and also to the {\it ergodicity} of the Laplacian (Theorem~ \ref{ergconect}), that in this context means that the only solutions of the equation $\Delta_m f = 0$ are the constant functions, recall further that this is, in turn, equivalent to the ergodicity of the measure $\nu$ (see also Theorem~\ref{ergo1}).  Moreover, we relate it with geometric properties of the metric random walk space (Theorem \ref{ERggB}).

In 1969 Jeff Cheeger \cite{Cheeger} proved his famous inequality
$$\frac{h^2_M}{2} \leq \lambda_1(\Delta_M),$$
where $\lambda_1(\Delta_M)$ is the first non-trivial eigenvalue of the Laplace Beltrami operator $\Delta_M$ on $L^2(M, {\rm vol})$
 of a compact manifold $M$ and the Cheeger constant $h_M$ is defined as
$$h_M = \inf \frac{\hbox{Area}(\partial S)}{\min (\hbox{vol}(S),\hbox{vol}(M \setminus S))},$$
where the infimum runs over all $S \subset  M$ with sufficiently smooth boundary. This inequality can be traced back to the paper by Polya and Szego \cite{PS}. The first  Cheeger estimates on graphs are due to Dodziuk \cite{D}
and Alon and Milmann \cite{AM}. Since then, these estimates have been improved and various variants have been proved. For locally finite weighted connected graphs, the following relation between the Cheeger constant and the first positive eigenvalue $\lambda_1(G)$ of the graph Laplacian  has been proved in \cite{Ch} (see also~\cite{BJ})
$$
 \frac{h_G^2}{2}\leq \lambda_1(G) \leq 2  h_G,
$$
where $h_G$ is the Cheeger constant for graphs.  For a  general metric random walk space $[X,d,m]$ we define the Cheeger constant $h_m(X)$ and we obtain the Cheeger inequality (Theorem~\ref{isopoin})
$$
 \frac{h_m^2}{2}\leq {\rm gap} (- \Delta_m) \leq 2  h_m,
$$
where ${\rm gap} (- \Delta_m)$ is the spectral gap of the Laplace operator. As a consequence,  we show that a Poincar\'{e} inequality holds if, and only if, an isoperimetric inequality holds.

An important tool in the study of the speed of convergence of the heat flow to the equilibrium is the Poincar\'{e} inequality (see \cite{BGL}). In the case of Riemannian manifolds and Markov diffusion semigroups, a usual condition required to obtain this functional inequality is the positivity of the corresponding Ricci curvature of the underlying space (see \cite{BGL}, \cite{Villani2}). In \cite{BE}, Bakry and Emery  found a way to define the lower Ricci curvature bound through the heat flow. Moreover, Renesse and Sturm  \cite{RenesseSturm}  proved that, on a Riemannian manifold $M$, the Ricci curvature is bounded from below by some constant $K \in \R$ if, and only if, the Boltzmann-Shannon entropy is $K$-convex along geodesics in the $2$-Wasserstein space of probability measures on $M$. This was the key observation, used simultaneously by Lott and Villani \cite{LV1} and  Sturm \cite{Sturm}, to give a notion of a lower Ricci curvature bound in the general context of length metric measure spaces. In these spaces, the relation between the Bakry-\'{E}mery curvature-dimension condition and the notion of the Ricci curvature bound introduced by Lott-Villani-Sturm,  was done by   Ambrosio, Gigli and  Savar\'{e} in~\cite{AGS-AnnPr}, where they proved that these two notions of Ricci curvature coincide under certain assumptions on the metric measure space.

When the space under consideration is discrete, for instance, in the case of a graph, the previous concept  of a Ricci curvature bound is not as clearly applicable as in the continuous setting.  Indeed, the definition by Lott-Sturm-Villani does not apply if the $2$-Wasserstein space over the metric measure space does not contain geodesics. Unfortunately, this is the case if the underlying space is discrete. Recently,  Erbas and  Maas \cite{EM}, in the framework of Markov chains on discrete spaces, in order to circumvent the nonexistence of $2$-Wasserstein geodesics, replace the $2$-Wasserstein metric by a different metric, which was introduced by Maas in \cite{Maas}. Here, we do not consider this notion of Ricci curvature bound which, in the framework of metric random walk spaces, will be the object of the forthcoming paper \cite{MST2}. Instead, we will use two other concepts of a Ricci curvature bound, the one based on the Bakry-\'{E}mery curvature-dimension condition and the one introduced by Y. Ollivier in \cite{O}. We refer to \cite{NR} and the references therein for the vibrant research field of discrete curvature.

The use of the Bakry-\'{E}mery curvature-dimension condition to obtain a possible definition of a Ricci curvature bound in Markov chains was first considered  in 1998 by Schmuckenschlager \cite{S}.  Moreover, in~2010,   Lin and  Yau \cite{LY} used this concept for graphs. Subsequently, this concept of curvature in the discrete setting has been frequently used (see \cite{KKRT} and the references therein). Note that, to deal with  the Bakry-\'{E}mery curvature-dimension condition, one needs   a {\it Carr\'{e} du champ} $\Gamma$. In the framework of Markov diffusion semigroups in order to get good inequalities from this curvature-dimension condition it is essential that the generator $A$ of the semigroup satisfies the chain rule formula
$$A(\Phi(f)) = \Phi^{\prime}(f) A(f) + \Phi^{\prime \prime}(f) \Gamma(f),$$
which characterizes diffusion operators in the continuous setting (see \cite{BGL}). Unfortunately, this chain rule does not hold in the discrete setting and this is one of the main difficulties when working with this curvature-dimension condition in metric random walk spaces.

In Riemannian geometry, positive Ricci curvature is characterized  by the fact that \lq\lq small balls are closer, in the $1$-Wasserstein distance, than their centers are" (see \cite{RenesseSturm}).  In the framework of metric random walk spaces, inspired by this, Y. Ollivier   \cite{O} introduced the concept of {\it coarse Ricci curvature}, substituting the balls by the measures $m_x$. Moreover, he proved that positive coarse Ricci curvature implies  positivity of the spectral gap. In Section~\ref{sect03}  we  give conditions on the Laplace operator $\Delta_m$ which ensure the positivity of the spectral gap and we relate bounds on the spectral gap with bounds on the Bakry-\'{E}mery curvature-dimension condition.

Following the papers by Marton and Talagrand (\cite{Marton}, \cite{Talagrand}) about transport inequalities that relate Wasserstein distances with entropy and information, this research topic  has had a great development (see the survey \cite{GL}). One of the keystones of this theory  was the discovery in 1986 by Marton \cite{Marton0} of the link between transport inequalities and the concentration of measure. Concentration of measure inequalities can be obtained by means of other functional inequalities such as isoperimetric and logarithmic Sobolev inequalities, see the textbook by Ledoux \cite{Ledoux} for an excellent account on the subject.  We show that under the positivity of the  Bakry-\'{E}mery curvature-dimension condition or the Ollivier-Ricci curvature a transport-information inequality holds (Theorems~\ref{transptinq1} and \ref{megustados}). Moreover, we prove that if a  transport-information inequality holds then a transport-entropy inequality is also satisfied (Theorem~\ref{transptinq11}) and that, in general, the converse implication  does not hold.

 \subsection{Metric Random Walk Spaces}

Let $(X,d)$ be a  Polish metric  space equipped with its Borel $\sigma$-algebra.
 \begin{definition}\rm
A {\it random walk} $m$ on $X$ is a family of probability measures $m_x$ on $X$, $x \in X$, satisfying the following two technical conditions:
\item{(i)} the measures $m_x$  depend measurably on the point  $x \in X$, i.e., for any Borel subset $A$ of $X$ and any Borel subset $B$ of $\R$, the set $\{ x \in X \ : \ m_x(A) \in B \}$ is Borel,
\item{(ii)} each measure $m_x$ has finite first moment, i.e. for some (hence any) $z \in X$, and for any $x \in X$ one has $\int_X d(z,y) dm_x(y) < +\infty$ (see~\cite{O}).

\noindent A {\it metric random walk  space} $[X,d,m]$  is a  Polish metric space $(X,d)$ equipped with a  random walk $m$.
\end{definition}

Let $[X,d,m]$ be a metric random walk  space. A   Radon measure $\nu$ on $X$ is {\it invariant} for the random walk $m=(m_x)$ if
$$d\nu(x)=\int_{X}d\nu(y)dm_y(x),$$
that is,  for any $\nu$-measurable set $A$, it holds that $A$ is $m_x$-measurable  for $\nu$-almost all $x\in X$,  $\displaystyle x\mapsto  m_x(A)$ is $\nu$-measurable and
$$\nu(A)=\int_X m_x(A)d\nu(x).$$
Hence, for any $f \in L^1(X, \nu)$, it holds that $f \in L^1(X, m_x)$ for $\nu$-a.e. $x \in X$, $\displaystyle x\mapsto \int_X f(y) d{m_x}(y)$ is $\nu$-measurable and
$$\int_X f(x) d\nu(x) = \int_X \left(\int_X f(y) d{m_x}(y) \right)d\nu(x).$$
Note that, following the notation in the introduction, $\nu$ is invariant if $\nu\ast m=\nu$.

The measure $\nu$ is said to be {\it reversible} if, moreover, the detailed balance condition \begin{equation}\label{repo001}
dm_x(y)d\nu(x)  = dm_y(x)d\nu(y)
\end{equation}
holds. Under suitable assumptions on the  metric random walk  space $[X,d,m]$, such an invariant and reversible measure $\nu$ exists and is unique, as we will see below. Note that the reversibility condition implies the invariance condition.

We will assume that the measure space $(X,\nu)$ is $\sigma$-finite.

 \begin{example}\label{JJ}{\rm
  \begin{enumerate}
   \item \label{dom001}
  Let $(\R^N, d, \mathcal{L}^N)$, with $d$ the Euclidean distance and $\mathcal{L}^N$ the Lebesgue measure. Let  $J:\R^N\to[0,+\infty[$ be a measurable, nonnegative and radially symmetric
function  verifying $\int_{\R^N}J(z)dz=1$. In $(\R^N, d, \mathcal{L}^N)$ we have the following random walk,
$$m^J_x(A) :=  \int_A J(x - y) d\mathcal{L}^N(y) \quad \hbox{ for every Borel set } A \subset  \R^N \hbox{ and }x\in\R^N.$$
Applying Fubini's Theorem it easy to see that the Lebesgue measure $\mathcal{L}^N$ is an invariant and  reversible measure for this random walk.

 \item  \label{dom002} Let $K: X \times X \rightarrow \R$ be a Markov kernel on a countable space $X$, i.e.,
 $$K(x,y) \geq 0 \quad \forall x,y \in X, \quad \quad \sum_{y\in X} K(x,y) = 1 \quad \forall x \in X.$$
 Then, for $$m^K_x(A):= \sum_{y \in A} K(x,y),$$
 $[X, d, m^K]$ is a metric random walk space for  any  metric  $d$ on $X$. For irreducible and positive recurrent Markov chains (see  for example~\cite{HLL})
 there exists  a unique  stationary probability measure (also called steady state) on $X$, that is,  a measure   $\pi $ on $X$ satisfying
 $$\sum_{x \in X} \pi(x) = 1 \quad \hbox{and} \quad \pi(y) = \sum_{x \in X} \pi(x) K(x,y) \quad \quad \forall y \in X.$$
This stationary probability measure $\pi$ is said to be reversible for $K$  if the following detailed balance equation
 $$K(x,y) \pi(x) = K(y,x) \pi(y)$$ holds for $x, y \in X$. By Tonelli's Theorem for series, this   balance condition is equivalent to the one given in~\eqref{repo001} for $\nu=\pi$:
 $$dm^K_x(y)d\pi(x)  =   dm^K_y(x)d\pi(y).$$

 \item \label{dom003} A weighted discrete graph $G = (V(G), E(G))$ is a graph with vertex set $V(G)$ and edge set $E(G)$ such that to each edge $(x,y) \in E(G)$ (we will write $x\sim y$ if $(x,y) \in E(G)$) we assign  a positive weight $w_{xy} = w_{yx}$. We consider that $w_{xy} = 0$ if $(x,y) \not\in E(G)$.
We say that a vertex $x\in V(G)$ is simple if it has no loops, so that $w_{xx}=0$. A graph is said to be simple if all the vertices are simple.

  A finite sequence $\{ x_k \}_{k=0}^n$  of vertices on a graph is called a {\it  path} if $x_k \sim x_{k+1}$ for all $k = 0, 1, ..., n-1$. The {\it length} of a path is defined as the number, $n$, of edges in the path.

A graph $G = (V(G), E(G))$ is called {\it connected} if, for any two vertices $x, y \in V$,
there is a path connecting $x$ and $y$, that is, a sequence of vertices $\{ x_k \}_{k=0}^n$ such that $x_0 = x$ and
$x_n = y$.  If $G = (V(G), E(G))$ is connected then define the graph distance $d_G(x,y)$ between any
two distinct vertices $x, y$ as   the minimum of the lengths  of the paths connecting $x$ and $y$.

 For each $x \in V(G)$  we define
 $$d_x:= \sum_{y \sim x} w_{xy}.$$

 When $w_{xy}=1$ for every $(x,y)\in E(G)$ with $x\sim y$, $d_x$ coincides with the degree of the vertex $x$ in the graph, that is,  the number of edges containing $x$. A graph $G = (V(G), E(G))$ is called {\it locally finite} if each vertex belongs to a finite number of edges.

  For each $x \in V(G)$  we define the following probability measure
$$m^G_x=  \frac{1}{d_x}\sum_{y \sim x} w_{xy}\,\delta_y.
$$
 If $G = (V(G), E(G))$ is a locally finite weighted   connected   graph, we have that $[V(G), d_G, (m^G_x)]$ is a metric random walk space. Furthermore, it is not difficult to see that the measure $\nu_G$ defined as
 $$\nu_G(A):= \sum_{x \in A} d_x,  \quad A \subset V(G),$$
is an invariant and  reversible measure for this random walk.

\item \label{dom006} From a metric measure space $(X,d, \mu)$ we can obtain a metric random walk space, the so called {\it $\epsilon$-step random walk associated to $\mu$}, as follows. Assume that balls in $X$ have finite measure and that ${\rm Supp}(\mu) = X$. Given $\epsilon > 0$, the $\epsilon$-step random walk on $X$, starting at point~$x$, consists in randomly jumping in the ball of radius $\epsilon$ around $x$, with probability proportional to $\mu$; namely
 $$m^{\mu,\epsilon}_x:= \frac{\mu \res B(x, \epsilon)}{\mu(B(x, \epsilon))}.$$
Note that $\mu$ is an invariant and reversible measure for the metric random walk space $[X, d, m^{\mu,\epsilon}]$.

\item \label{dom00606} Given a  metric random walk  space $[X,d,m]$ with invariant and reversible measure $\nu$ for $m$, and given a $\nu$-measurable set $\Omega \subset X$ with $\nu(\Omega) > 0$, if we define, for $x\in\Omega$,
$$m^{\Omega}_x(A):=\int_A d m_x(y)+\left(\int_{X\setminus \Omega}d m_x(y)\right)\delta_x(A) \quad \hbox{ for every Borel set } A \subset  \Omega  ,
$$
 we have that $[\Omega,d,m^{\Omega}]$ is a metric random walk space and it easy to see that $\nu \res \Omega$ is  reversible for $m^{\Omega}$.
 \end{enumerate}
}
\end{example}

Given a metric random walk  space $[X,d,m]$, geometrically we may think of $m_x$  as a replacement for the notion of balls around $x$, while in probabilistic terms we can
rather think of these data as defining a Markov chain whose transition probability from $x$ to $y$ in $n$
steps is
\begin{equation}\label{RW1}
\displaystyle
dm_x^{*n}(y):= \int_{z \in X}  dm_z(y)dm_x^{*(n-1)}(z)
\end{equation}
where $m_x^{*1} = m_x$. Note that $m_x^{*n}=m_x^{*(n-1)}\ast m_x$ for any $x\in X$.

Observe that
$$\int_{y\in X}f(y)dm_x^{*n}(y)=\int_{z \in X}\left(\int_{y\in X}f(y)dm_z(y)\right)dm_x^{*(n-1)}(z).$$
Thus, inductively,
$$\int_{y \in X} dm_x^{*n}(y)=\int_{z \in X}\left(\int_{y\in X}  dm_z(y)\right)dm_x^{*(n-1)} (z)=\int_{z \in X}dm_x^{*(n-1)} (z)=1.$$
Hence,
  $[X, d, m^{*n}]$ is also a metric random walk  space.  Moreover, if  $\nu$ is invariant and reversible for $m$, then $\nu$ is  also invariant and reversible  for $m^{*n}$.

\begin{definition} Let $[X,d,m]$ be a metric random walk  space. We say that $[X,d,m]$ has the {\it strong-Feller property} if
$$ m_{x_0}(A) = \lim_{n \to +\infty} m_{x_n}(A) \quad \hbox{for every Borel set } \ A \subset X$$
whenever $x_n \to x_0$ as $n \to +\infty$ in $(X,d)$.
\end{definition}

 Note that the examples of metric random walk  spaces given in   Example \ref{JJ} have the strong-Feller property.

 In \cite{MST1} we study the concepts of {\it $m$-perimeter} and {\it $m$-mean curvature} associated with a metric random walk  space $[X,d,m]$ with invariant and reversible measure $\nu$  with respect to $m$.  For this aim, we introduce the notion of nonlocal interaction between two $\nu$-measurable subsets $A$ and $B$ of $X$  as
$$ L_m(A,B):= \int_A \int_B dm_x(y) d\nu(x).
$$
For $L_m(A,B) < +\infty$, by the reversibility assumption on $\nu$, we have that
 $$L_m(A,B)=L_m(B,A).$$
We then define the concept of $m$-perimeter of a $\nu$-measurable subset $E \subset X$ as
 $$P_m(E) = L_m(E,X\setminus E)=\int_E \int_{X\setminus E} dm_x(y) d\nu(x).$$
If $\nu(E)<+\infty$, we have
 \begin{equation}\label{secondf021}\displaystyle P_m(E)=\nu(E) -\int_E\int_E dm_x(y) d\nu(x).
\end{equation}
It is easy to see that, on account of the reversibility of $\nu$,
$$
P_m(E) = \frac{1}{2} \int_{X}  \int_{X}  \vert \1_{E}(y) - \1_{E}(x) \vert dm_x(y) d\nu(x).
$$

In the particular case of  a graph $[V(G), d_G, m^G ]$, the definition of perimeter of a set $E \subset V(G)$ is given by
$$\vert \partial E \vert := \sum_{x \in E, y \in V \setminus E} w_{xy}.$$
Then, we have that
\begin{equation}\label{perim}
\vert \partial E \vert = P_{m^G}(E) \quad \hbox{for all} \ E \subset V(G).
\end{equation}

  In \cite{MST1}, we also introduce the $m$-total variation of a function $u : X \rightarrow \R$ as
$$TV_m(u):= \frac{1}{2} \int_{X}  \int_{X}  \vert u(y) - u(x) \vert dm_x(y) d\nu(x)$$
 and we prove the following {\it Coarea formula}. Note that
 $$P_m(E) = TV_m(\1_E).$$
\begin{theorem}\label{coarea11} (\cite[Theorem 2.7]{MST1}) For any $u \in L^1(X,\nu)$, let $E_t(u):= \{ x \in X \ : \ u(x) > t \}$. Then
$$
TV_m(u) = \int_{-\infty}^{+\infty} P_m(E_t(u))\, dt.
$$
\end{theorem}

Let $E \subset X$ be $\nu$-measurable. For a point $x  \in X$ we define  the {\it $m$-mean curvature of $\partial E$ at $x$} as
$$\mathcal{H}^m_{\partial E}(x):= \int_{X}  \Big(\1_{X \setminus E}(y) - \1_E(y)\Big) dm_x(y) = 1 - 2 \int_E  dm_x(y).$$
Note that $H^m_{\partial E}(x)$ can be computed for every $x \in X$, not only for points in $\partial E$. Furthermore, for a $\nu$-integrable set~$E$,
$$\int_E \mathcal{H}^m_{\partial E}(x) d\nu(x) = \int_E \left( 1 - 2 \int_E  dm_x(y) \right)  d\nu(x) = \nu(E) - 2\int_E\int_E dm_x(y) d\nu(x),$$
hence, having in mind \eqref{secondf021}, we obtain that
 \begin{equation}\label{1secondf021}\displaystyle \int_E \mathcal{H}^m_{\partial E}(x) d\nu(x)=2P_m(E) -\nu(E).
\end{equation}

\subsection{Ollivier-Ricci Curvature}

Let $(X,d)$ be a Polish metric space and $\mathcal{M}^+(X)$ the set of positive Radon measures on $X$.
 Fix  $\mu, \nu \in \mathcal{M}^+(X)$ satisfying the mass balance condition
\begin{equation}\label{massB}
\mu(X) = \nu (X).
\end{equation}
The Monge-Kantorovich problem is the minimization problem
$$\min  \left\{
\int_{X \times X} d (x,y) \, d\gamma(x,y) \, : \, \gamma
\in \Pi(\mu, \nu) \right\},$$ where $\Pi(\mu, \nu):= \left\{
\hbox{Radon measures} \ \gamma \ \hbox{in} \ X \times X
:   \pi_0\# \gamma = \mu,    \pi_1\# \gamma = \nu \right\}$, with $\pi_\alpha(x,y):= x + \alpha(y -x)$ for $\alpha\in\{0,1\}$.

For $1 \leq p < \infty$, the {\it $p$-Wasserstein distance} between $\mu, \nu $ is defined as
$$
W_p^{d}(\mu, \nu):= \left(\min  \left\{
\int_{X \times X} d (x,y)^p \, d\gamma(x,y) \, : \, \gamma
\in \Pi(\mu, \nu) \right\} \right)^{\frac{1}{p}}.
$$

The Monge-Kantorovich problem has a dual formulation that can be
stated in this case as follows (see for instance \cite[Theorem
1.14]{Villani}).

{\bf Kantorovich-Rubinstein's Theorem.} {\it Let $\mu, \nu \in
\mathcal{M}^+(X)$ be two measures satisfying the mass balance
condition \eqref{massB}.  Then,
$$
\begin{array}{l}W_1^{d}(\mu, \nu) =
\displaystyle\sup \left\{
 \int_{X} u \, d(\mu - \nu) \,  : \,   u \in K_{d}(X) \right\}
 \\[12pt]
 \phantom{W_1^{d}(\mu, \nu)}
 = \displaystyle\sup \left\{
 \int_{X} u \, d(\mu - \nu) \,  : \,   u \in K_{d}(X) \cap L^\infty(X,\nu)\right\}
 \end{array}
$$
where
$$
K_{d}(X) := \left\{ u:X \mapsto \R \, : \,  | u(y) - u(x)|
\leq d(y,x)   \right\}.
$$
}

In \cite{O} Y. Ollivier gives the following definition of coarse Ricci curvature that we will call Ollivier-Ricci curvature.

\begin{definition}[\cite{O}]\label{defRicc}{\rm On a given metric random walk  space $[X,d,m]$, for any two distinct points $x,y \in X$, the {\it Ollivier-Ricci curvature  of $[X,d,m]$ along $(x,y)$} is defined as
$$\kappa_m(x,y):= 1 - \frac{W_1^d(m_x,m_y)}{d(x,y)}.$$
 The {\it Ollivier-Ricci curvature  of $[X,d,m]$} is defined by
$$\kappa_m:= \inf_{\tiny \begin{array}{c}x,y \in X\\  x \not= y\end{array}} \kappa_m(x,y).$$
We will write $\kappa(x,y)$ instead of $\kappa_m(x,y)$, and $\kappa =\kappa_m$, if the context allows no confusion.
}
\end{definition}

In the case that $(X,d, \mu)$ is a smooth complete Riemannian manifold, if $(m^{\mu,\epsilon}_x)$ is the  $\epsilon$-step random walk associated to $\mu$ given in Example \ref{JJ} \eqref{dom006}, then it is proved in \cite{RenesseSturm} (see also  \cite{O}) that $\kappa_{m^{\mu,\epsilon}}(x,y)$ gives back the ordinary Ricci curvature when $\epsilon \to 0$, up to scaling by $\epsilon^2$.

\begin{example}\label{Olivier1}{\rm Let $[\R^N, d, m^J]$ be the metric random walk space given in Example \ref{JJ} (\ref{dom001}). Let us see that  $\kappa(x,y) = 0$. Given $x, y \in \R^N$, $x \not= y$, by Kantorovich-Rubinstein's Theorem, we have
$$W_1^d(m^J_x, m^J_y) =
   \sup \left\{
 \int_{\R^N} u(z) (J(x-z) - J(y- z)) \, dz\,  : \, u \in K_{d}(\R^N)  \right\}$$ $$= \sup \left\{
 \int_{\R^N}( u(x+z) - u(y +z)) J(z) \, dz\,  : \, u \in K_{d}(\R^N)  \right\}.$$
 Now, for $u \in K_{d}(\R^N)$, we have
$$\int_{\R^N}( u(x+z) - u(y +z)) J(z) \, dz \leq \Vert x - y \Vert.  $$
Thus, $W_1^d(m^J_x, m^J_y) \leq \Vert x - y \Vert$. On the other hand, taking $u(z):= \frac{\langle z, x-y \rangle}{\Vert x - y \Vert}$, we have $u \in K_{d}(\R^N)$, hence
$$W_1^d(m^J_x, m^J_y) \geq  \int_{\R^N}( u(x+z) - u(y +z)) J(z) \, dz = \Vert x - y \Vert.$$
Therefore,
$$W_1^d(m^J_x, m^J_y)  = \Vert x - y \Vert,$$
and, consequently, $\kappa(x,y) = 0$.
}
\end{example}

 \begin{example}\label{1Olivier1}{\rm
 Let $[V(G), d_G, (m^G_x)]$ be the  metric random walk space associated to the locally finite weighted discrete graph $G = (V(G), E(G))$ given in Example~\ref{JJ}~(\ref{dom003}) and let $N_G(x):=\{z\in V(G) : z\sim x \}$ for $x\in V(G)$. Then, the Ollivier-Ricci curvature along $(x,y) \in E(G)$ is
$$\kappa(x,y) = 1 - \frac{W_1^{d_G}(m_x,m_y)}{d_G(x,y)},$$
where
$$W_1^{d_G}(m_x,m_y) =  \inf_{\mu \in \mathcal{A}} \sum_{z_1 \sim x}\, \sum_{z_2 \sim y} \mu(z_1,z_2) d_G(z_1,z_2),
$$
being $\mathcal{A}$  the set of all  matrices with entries indexed by $N_G(x) \times N_G(y)$ such that $\mu(z_1, z_2) \geq 0$ and
$$\sum_{z_2 \sim y} \mu(z_1, z_2) = \frac{w_{xz_1}}{d_x}, \quad \sum_{z_1 \sim x} \mu(z_1, z_2) = \frac{w_{yz_2}}{d_y},\quad\hbox{for } (z_1,z_2)\in N_G(x) \times N_G(y).$$
 }
 \end{example}
There is an extensive literature about Ollivier-Ricci curvature on discrete graphs  (see for instance, \cite{BM},  \cite{BJL}, \cite{ChP}, \cite{GRST}, \cite{JL}, \cite{LY}, \cite{O}, \cite{OS}, \cite{OV} and \cite{Paeng}).

\section{The Heat Flow on Metric Random Walk  Spaces}

\subsection{The Heat Flow}

Let $[X,d,m]$ be a metric random walk  space with   invariant  measure $\nu$ for $m$. For a function $u : X \rightarrow \R$ we define its {\it nonlocal gradient} $\nabla u: X \times X \rightarrow \R$ as
$$\nabla u (x,y):= u(y) - u(x) \quad \forall \, x,y \in X,$$
and for a function $\z : X \times X \rightarrow \R$, its {\it $m$-divergence} ${\rm div}_m \z : X \rightarrow \R$ is defined as
 $$({\rm div}_m \z)(x):= \frac12 \int_{X} (\z(x,y) - \z(y,x)) dm_x(y).$$

The {\it averaging operator} on $[X,d,m]$ (see, for example, \cite{O}) is defined as
$$M_m f(x):= \int_X f(y) dm_x(y),$$
 when this expression has sense,
and the {\it Laplace operator} as $\Delta_m:= M_m - I$, i.e.,
$$\Delta_m f(x)= \int_X f(y) dm_x(y) - f(x) = \int_X (f(y) - f(x)) dm_x(y).$$
Note that
$$\Delta_m f (x) = {\rm div}_m (\nabla f)(x)$$
 and $\left(M_m\right)^n=M_{m^{\ast n}}$ for $n\in\N$.

 Due to the invariance of $\nu$ for the random walk $m$, both operators are well defined from  $L^1(X,\nu)$ to $L^1(X,\nu)$, $\Vert M_mf\Vert_1\le \Vert f\Vert_1$ and $\Vert \Delta_mf\Vert_1\le \Vert f\Vert_1$. Moreover, they map functions which are pointwise bounded by $C>0$ into functions pointwise bounded by $C$.
Observe that the invariance of $\nu$ can be rewritten as the following property:
\begin{equation}\label{Lap0}
\int_X \Delta_m f(x) d\nu(x) = 0 \quad \hbox{$\forall\, f\in L^1(X,\nu)$}.
\end{equation}

In the case of the weighted discrete graph $G$ with the random walk defined in Example~\ref{JJ}~(\ref{dom003}), the above operator is the graph Laplacian studied by many authors (see e.g. \cite{BJ}, \cite{BJL}, \cite{DK} or  \cite{JL}).

  By Jensen's inequality, we have that,  for $f \in L^2(X, \nu)\cap  L^1(X, \nu)$,
$$\begin{array}{c}\displaystyle \Vert M_m f \Vert^2_{L^2(X,\nu)} = \int_X \left( \int_X f(y) dm_x(y) \right)^2 d\nu(x)\\[12pt]
\displaystyle \leq  \int_X   \int_X f^2(y) dm_x(y) d\nu(x) =  \int_X f^2(x) d\nu(x) =\Vert f \Vert^2_{L^2(X,\nu)}.
\end{array}$$
Therefore, $M_m$ and $\Delta_m$ are linear operators in $L^2(X,\nu)$  with  domain $$D(M_m) = D(\Delta_m) = L^2(X, \nu) \cap  L^1(X, \nu).$$ Moreover,  in the case $\nu(X) < +\infty$,  $M_m$ and $\Delta_m$ are bounded  linear operators in $L^2(X,\nu)$  satisfying $\Vert M_m  \Vert  \leq 1$ and $\Vert \Delta_m \Vert  \leq 2$.

 If the invariant measure $\nu$ is reversible,   the following {\it  integration by parts formula} is straightforward:
 \begin{equation}\label{intbpart}
 \int_X f(x) \Delta_m g (x) d\nu(x) =  -\frac {1}{2} \int_{X \times X} (f(y)-f(x))  (g(y) - g(x)) dm_x(y)  d\nu(x)
 \end{equation}
 for $f,g \in L^2(X, \nu)\cap  L^1(X, \nu)$.

 In $L^2(X, \nu)$ we consider the symmetric form given by
$$
\mathcal{E}_m(f,g)   = - \int_X f(x) \Delta_mg (x) d\nu(x)  = \frac{1}{2} \int_{X \times X}\nabla f(x,y)\nabla g(x,y) d{m_x}(y)d\nu(x),
$$
with domain  for both variables $D(\mathcal{E}_m) = L^2(X, \nu)\cap  L^1(X, \nu)$, which is a linear and dense subspace of $L^2(X,\nu)$.

 Recall the definition of generalized product $\nu \otimes m_x$ (see, for instance, \cite[Definition 2.2.7]{AFP}), which is defined as the measure in $X \times X$ such that
$$\int_{X \times X} g(x,y) d(\nu \otimes m_x)(x,y) := \int_X \left( \int_X g(x,y) dm_x(y) \right) d\nu(x)$$
for every bounded Borel function $g$ with ${\rm supp}(g) \subset A \times B$, $A, B \subset \subset X$. In the previous definition we need to assume that the map $x \mapsto m_x(E)$ is $\nu$-measurable for any Borel set $E \in \mathcal{B}(X)$. Note that we can write
$$\mathcal{E}_m(f,g)    = \frac{1}{2} \int_{X \times X}\nabla f(x,y)\nabla g(x,y) d(\nu \otimes m_x)(x,y).$$

\begin{theorem}\label{generator} Let $[X,d,m]$ be a metric random walk  space with   invariant and reversible measure $\nu$ for $m$. Then, $- \Delta_m$ is a non-negative self-adjoint operator in $L^2(X, \nu)$ with  associated  closed symmetric form $\mathcal{E}_m$, which, moreover,
 is a  Markovian form.
\end{theorem}

\begin{proof} For $f \in D(\Delta_m)$, by the integration by parts formula~\eqref{intbpart}, we have
$$\int_X f(x) (- \Delta_m  f)(x) d\nu(x) =  \mathcal{E}_m(f,f) \geq 0.$$
Also, as a consequence of \eqref{intbpart}, we have that $- \Delta_m$ is a self-adjoint operator in $L^2(X, \nu)$.

To prove the closedness of $\mathcal{E}_m$, consider $f_n \in D(\mathcal{E}_m)$ such that
$$\mathcal{E}_m(f_n - f_k, f_n - f_k) \to 0 , \quad \hbox{when} \ n,k \to +\infty,$$
and
$$\Vert f_n - f_k \Vert_{L^2(X, \nu)}   \to 0 , \quad \hbox{when} \ n,k \to +\infty.$$
Since $f_n \to f$ in $L^2(X, \nu)$, we can assume that there exists a $\nu$-null set $N$ such that $f_n(x) \to f(x)$ for all $x \in X \setminus N$. Then, $(f_n(x) - f_n(y))^2 \to (f(x) - f(y))^2$ for all $(x,y) \in (X \setminus N) \times (X\setminus N) = (X \times X) \setminus[(N \times X) \cup (X \times N)]$. Now, since $\nu$ is invariant, we have
$$\nu \otimes m_x([(N \times X) \cup (X \times N)]) = \int_N \left(\int_X dm_x(y) \right) d\nu(x) + \int_X \left(\int_X \1_N(y) dm_x(y) \right) d\nu(x)$$  $$=\nu(N) + \int_{X} \1_{N}(y) d\nu(y) = 2  \nu(N) = 0.$$
Then, by Fatou's Lemma we have
$$\lim_{n \to \infty} \mathcal{E}_m(f_n - f, f_n - f) = \lim_{n \to +\infty} \frac{1}{2} \int_{X \times X}(\nabla (f_n- f)(x,y))^2 d(\nu \otimes m_x)(x,y) $$ $$=\lim_{n \to +\infty} \frac{1}{2} \int_{X \times X} \liminf_{k \to +\infty}(\nabla (f_n- f_k)(x,y))^2 d(\nu \otimes m_x)(x,y)  $$ $$ \leq \lim_{n \to +\infty} \liminf_{k \to +\infty} \frac{1}{2} \int_{X \times X} (\nabla (f_n- f_k)(x,y))^2 d(\nu \otimes m_x)(x,y) = 0.$$
Therefore, $ \mathcal{E}_m$ is closed. Moreover, for every $1$-Lipschitz map $\eta : \R \rightarrow \R$ with $\eta(0) = 0$, we have
$$\mathcal{E}_m(\eta \circ f, \eta \circ f) \leq \mathcal{E}_m( f, f)\quad \hbox{for every} \  f \in D(\mathcal{E}_m),$$
 and, hence, $\mathcal{E}_m$ satisfies the Markov property.
\end{proof}

By Theorem~\ref{generator}, as a consequence of the theory developed in \cite[Chapter 1]{FOT}, we have that if $(T^m_t)_{t \geq 0}$ is the strongly continuous semigroup associated with $\mathcal{E}_m$, then $(T^m_t)_{t \geq 0}$ is a positivity preserving (i.e., $T^m_t f \geq 0$ if $f \geq 0$) Markovian semigroup (i.e., $0 \leq T^m_t f \leq 1$ $\nu$-a.e. whenever $f \in L^2(X, \nu)$, $0 \leq f \leq 1$ $\nu$-a.e.). Moreover, $\Delta_m$ is the infinitesimal generator of $(T^m_t)_{t \geq 0}$, that is $$\Delta_m f = \lim_{t\downarrow 0} \frac{T^m_t f - f}{t}, \quad \quad \forall \, f \in D(\Delta_m).$$

From now on we denote $e^{t\Delta_m}:= T^m_t$ and  we call $\{e^{t\Delta_m} \, : \, t \geq 0 \}$ the {\it heat flow on the metric random walk space} $[X,d,m]$  with invariant and reversible measure $\nu$ for $m$. For every $u_0 \in L^2(X, \nu)$, $u(t):= e^{t\Delta_m}u_0 $ is the unique solution of the heat equation
\begin{equation}\label{CP2}
\left\{ \begin{array}{ll} \frac{du}{dt}(t) = \Delta_m u(t) \quad \hbox{for every } t\in(0, +\infty), \\[10pt]  u(0) = u_0, \end{array}\right.
\end{equation}
in the sense that $u \in C([0,+\infty): L^2(X, \nu)) \cap C^1((0,+\infty): L^2(X, \nu))$ and verifies \eqref{CP2},
or equivalently,
\begin{equation}\label{CPNLgene}
\left\{ \begin{array}{ll} \displaystyle\frac{du}{dt}(t,x) = \displaystyle\int_{X} (u(t)(y)- u(t)(x)) dm_x(y) \quad \hbox{for every } t>0 \hbox{ and $\nu$-a.e. } x\in X, \\[12pt]  u(0) = u_0. \end{array}\right.
\end{equation}
By the Hille-Yosida exponential formula we have that
$$e^{t\Delta_m}u_0  = \lim_{n \to +\infty} \left[ \left(I - \frac{t}{n} \Delta_m \right)^{-1} \right]^n u_0.$$

As a consequence of \eqref{Lap0}, if $\nu(X) < +\infty$,  we have that the semigroup $(e^{t\Delta_m})_{t \geq 0}$ conserves the mass. In fact
$$\frac{d}{dt} \int_X e^{t\Delta_m}u_0(x) d\nu(x) = \int_X \Delta_m u_0(x) d\nu(x) = 0,
$$
and, therefore,
\begin{equation}\label{consermass}
\int_X e^{t\Delta_m}u_0(x) d\nu(x) = \int_X u_0(x) d\nu(x).
\end{equation}

Associated with $\mathcal{E}_m$   we define the energy functional
$$\mathcal{H}_m(f)  := \mathcal{E}_m(f,f),$$
that is, $\mathcal{H}_m : L^2(X, \nu) \rightarrow [0, + \infty]$ is defined as
$$\mathcal{H}_m(f):= \left\{ \begin{array}{ll} \displaystyle\frac{1}{2} \int_{X \times X} (f(x) - f(y))^2 dm_x(y) d\nu(x) \quad &\hbox{ if $f\in L^2(X, \nu) \cap  L^1(X, \nu)$,} \\ \\ + \infty \quad &\hbox{else}. \end{array}\right.$$
We denote
$$D(\mathcal{H}_m):=L^2(X, \nu) \cap  L^1(X, \nu).$$
Note that, for $f\in D(\mathcal{H}_m)$, we have
$$\mathcal{H}_m(f) =   - \int_X f(x) \Delta_m f (x) d\nu(x).
$$

\begin{remark}\label{conntract1}{\rm It is easy to see that the functional $\mathcal{H}_m$ is convex and, moreover, with a proof similar to the proof of closedness in Theorem \ref{generator}, we get that the functional $\mathcal{H}_m$ is  closed and lower semi-continuous in $L^2(X, \nu)$. Now, it is not difficult to see that  $\partial \mathcal{H}_m = - \Delta_m$. Consequently, $- \Delta_m$ is a maximal monotone operator in $L^2(X, \nu)$.
We can also consider the heat flow in $L^1(X,\nu)$. Indeed, if we define in  $L^1(X,\nu)$ the operator $A$ as $Au = v \iff v(x) = - \Delta_mu(x)$ for all $x \in X$, then $A$ is a completely accretive operator. In fact, let
 $$\mathcal{P}:=\{ q \in C^\infty (\R) \ : \ 0 \leq q' \leq 1, \ \hbox{supp}(q') \ \hbox{is compact and} \ 0\not\in \hbox{supp}(q)\}.$$
 Given $f \in L^1(X, \nu)$, and $q\in \mathcal{P}$, applying \eqref{intbpart}, we have
$$\displaystyle \int_X q(f(x)) Af(x) d\nu(x) =  \frac {1}{2} \int_{X \times X} (q(f(y)) -q(f(x)))  (f(y) - f(x)) dm_x(y)  d\nu(x) \geq 0.$$
 Then, by  \cite[Proposition 2.2]{BCr2}, we have  that $A$ is a completely accretive operator. Moreover, $\overline{\partial \mathcal{H}_m}^{L^1(X,\nu)} = A$, thus $A$ is $m$-completely accretive in $L^1(X,\nu)$. Therefore, $A$ generates a $C_0$-semigroup $(S(t))_{t \geq 0}$ in $L^1(X,\nu)$ (see \cite{BrezisAF}) such that $S(t)f = e^{t\Delta_m} f$ for all $f \in L^1(X,\nu) \cap L^2(X,\nu)$, verifying
\begin{equation}\label{CAcont}
\Vert S(t)u_0 \Vert_{L^p(X, \nu)} \leq \Vert u_0 \Vert_{L^p(X, \nu)} \quad \forall u_0 \in L^p(X, \nu) \cap L^1(X, \nu), \quad 1 \leq p \leq +\infty .
\end{equation}
In the case that $\nu(X) < \infty$, we have that $S(t)$ is an extension to $L^1(X, \nu)$ of the heat flow $e^{t \Delta_m}$ in $L^2(X, \nu)$, that we will denote equally.
}
\end{remark}

\begin{example}\label{Jheat}{\rm

\begin{itemize}
\item[(1)]
Consider the metric random walk space $[X, d, m^K]$ associated with the Markov kernel $K$ (see Example \ref{JJ} (\ref{dom002})) and assume that the stationary probability measure $\pi$ is reversible. Then, the Laplacian $\Delta_{m^K}$ is given by
$$\Delta_{m^K} f(x):= \int_{X} f(y) dm^K_x(y) - f(x) = \sum_{y \in X} K(x,y)f(y) - f(x) \quad \forall f \in L^2(X, \pi).$$
Consequently, given $u_0\in L^2(X,\pi)$, $u(t):= e^{t\Delta_{m^K}}u_0$ is the solution of the  equation
$$
\left\{ \begin{array}{ll} \displaystyle\frac{du}{dt}(t,x) = \displaystyle\sum_{y \in X} K(x,y) u(t)(y)  - u(t)(x)\quad \hbox{on} \ \ (0, +\infty), \\[12pt]  u(0) = u_0. \end{array}\right.
$$
Therefore, $e^{t\Delta_{m^K}} = e^{t(K -I)}$ is the {\it heat semigroup} on $X$ with respect
to the geometry determined by the Markov kernel $K$. In the case that $X$ is a finite set, we have
$$e^{t\Delta_{m^K}} = e^{t(K -I)}= e^{-t} \sum_{n=0}^{+\infty} \frac{t^n K^n}{n!}.
$$

\item[(2)] If we consider the metric random walk space $[\R^N, d, m^J]$, being $m^J = (m^J_x)$ the random walk defined in Example \ref{JJ} (\ref{dom001}), we have that, for the invariant measure $\nu = \mathcal{L}^N$, the Laplacian is given by
$$\Delta_{m^J} f(x):= \int_{\R^N} (f(y)-f(x)) J(x -y) dy.$$
Then, given $u_0\in L^2(\R^N,\mathcal{L}^N)$ we have that $u(t):= e^{t\Delta_{m^J}}u_0$ is the solution of the $J$-nonlocal heat equation
\begin{equation}\label{CPNL}
\left\{ \begin{array}{ll} \displaystyle\frac{du}{dt}(t,x) = \displaystyle\int_{\R^N} (u(t)(y)- u(t)(x)) J(x -y) dy \quad \hbox{in} \ \ \R^N\times(0, +\infty), \\[12pt]  u(0) = u_0. \end{array}\right.
\end{equation}

In the case that $\Omega$ is a closed bounded subset of $\R^N$, if we consider the metric random walk space $[\Omega, d, m^{J,\Omega}]$, being $m^{J,\Omega} = (m^J)^{\Omega}$ (see Example~\ref{JJ}~\eqref{dom00606}), that is
$$m^{J,\Omega}_x(A):=\int_A J(x-y)dy+\left(\int_{\R^n\setminus \Omega}J(x-z)dz\right)\delta_x(A) \quad \hbox{ for every Borel set } A \subset  \Omega ,$$
we have that
$$\Delta_{m^{J,\Omega}} f(x) = \int_\Omega (f(y) - f(x)) dm^{J,\Omega}_x(y) =  \int_\Omega J(x - y) (f(y) - f(x))dy.$$
Then we have that $u(t):= e^{t\Delta_{m^{J,\Omega}}}u_0$ is the solution of the  homogeneous Neumann problem for the $J$-nonlocal heat equation:
\begin{equation}\label{NCPNL}
\left\{ \begin{array}{ll} \displaystyle\frac{du}{dt}(t,x) = \displaystyle\int_{\Omega} (u(t)(y)- u(t)(x)) J(x -y) dx \quad \hbox{in} \ \ (0, +\infty)\times \Omega, \\[12pt]  u(0) = u_0. \end{array}\right.
\end{equation}
See \cite{ElLibro} for a comprehensive study of   problems \eqref{CPNL} and \eqref{NCPNL}.

Observe that, in general, for a bounded set $\Omega\subset X$, and by using $m^\Omega$, we have that $u(t):= e^{t\Delta_{m^{\Omega}}}u_0$ is the solution of
$$
\left\{ \begin{array}{ll} \displaystyle\frac{du}{dt}(t,x) = \displaystyle\int_{\Omega} (u(t)(y)- u(t)(x)) dm_x(y) \quad \hbox{in} \ \ (0, +\infty)\times \Omega, \\[12pt]  u(0) = u_0, \end{array}\right.
$$
that, like~\eqref{NCPNL}, is an
   homogeneous  Neumann problem for the  $m$-heat equation.
\end{itemize}
}
\end{example}

In \cite{MRTHC},  it is shown, by means of the Fourier transform, that if $D \subset \R^N$ has $\mathcal{L}^N$-finite measure, then
\begin{equation}\label{Pheat}
e^{\Delta_{m^J}} \1_D (x) = e^{-t}\sum_{n=0}^{\infty}\int_{D} (J*)^n(x-y) dy\frac{t^n}{n!}.
\end{equation}
 In the next result we generalize \eqref{Pheat} for general metric random walk spaces.
    We use the   notation introduced in~\eqref{RW1}.

 \begin{theorem}\label{expannsion1} Let $[X, d, m]$ be a metric random walk space with invariant and reversible  measure $\nu$. Let $u_0\in L^2(X, \nu)\cap L^1(X, \nu)$.
 Then,
 \begin{equation}\label{d1704}
 e^{t\Delta_{m}} u_0(x) =  e^{-t}\left(u_0(x)+\sum_{n=1}^{+\infty}\int_{X} u_0(y)dm_x^{\ast n}(y)\frac{t^n}{n!}\right) = e^{-t}\sum_{n=0}^{+\infty}\int_{X} u_0(y)dm_x^{\ast n}(y)\frac{t^n}{n!},
 \end{equation}
 where $\displaystyle\int_{X} u_0(y)dm_x^{\ast 0}(y)=u_0(x)$.

 In particular, for $D \subset X$ with $\nu(D) < +\infty$, we have
$$
 e^{t\Delta_{m}} \1_D (x) =   e^{-t}\left(\1_D(x)+\sum_{n=1}^{+\infty}\int_{D} dm_x^{\ast n}(y)\frac{t^n}{n!}\right) =e^{-t}\sum_{n=0}^{+\infty} m_x^{\ast n}(D)\frac{t^n}{n!},
$$
   where $m_x^{\ast 0}(D) = \1_D(x)$.\label{pag1501}
\end{theorem}

\begin{proof} We define
$$u(x,t)= e^{-t}\left(u_0(x)+\sum_{n=1}^{+\infty}\int_{X} u_0(y)dm_x^{\ast n}(y)\frac{t^n}{n!}\right).$$
Note that, since $u_0\in L^1(X,\nu)$, then $u_0\in L^1(X,m_x^{\ast n})$ for $\nu$-a.e. $x\in X$ and every $n\in\N$, and
$$\int_X\sum_{n=0}^{k}\left|\int_{X} u_0(y)dm_x^{\ast n}(y)\right|\frac{t^n}{n!}d\nu(x)=\sum_{n=0}^{k}\int_X\left|\int_{X} u_0(y)dm_x^{\ast n}(y)\right|\frac{t^n}{n!}d\nu(x)$$
$$\leq\sum_{n=0}^{k}\int_X\int_{X} \left|u_0(y)\right|dm_x^{\ast n}(y)d\nu(x)\frac{t^n}{n!}=\sum_{n=0}^{k}\int_X \left|u_0(x)\right|d\nu(x)\frac{t^n}{n!}\leq e^t \Vert u_0 \Vert_{L^1(X,\nu)} \, .$$
Let
$$f_k(x)=\sum_{n=0}^{k}\left|\int_{X} u_0(y)dm_x^{\ast n}(y)\right|\frac{t^n}{n!}$$
then $0\leq f_k(x)\leq f_{k+1}(x) <+\infty$ and $\int f_k d \nu \leq e^t \Vert u_0 \Vert_{L^1(\nu)}$ for every $k\in\N$ so we may apply monotone convergence to get that
$$\int_X\sum_{n=0}^{+\infty}\left|\int_{X} u_0(y)dm_x^{\ast n}(y)\right|\frac{t^n}{n!}d\nu(x)\leq e^t \Vert u_0 \Vert_{L^1(X,\nu)}, $$
thus the function
$$x \mapsto \sum_{n=0}^{+\infty}\left|\int_{X} u_0(y)dm_x^{\ast n}(y)\right|\frac{t^n}{n!}$$
belongs to $L^1(X, \nu)$ and, consequently, is finite $\nu$-a.e. Note that the same is true for the function
$$x \mapsto \sum_{n=0}^{+\infty}\int_{X} \left|u_0(y)\right|dm_x^{\ast n}(y)\frac{t^n}{n!} \, .$$
From this we get that $u(x,t)$ is well defined and also the uniform convergence of the series for $t$ in compact subsets of $[0, +\infty)$. Hence,
$$\frac{du}{dt}(x,t)=-u(x,t)+ e^{-t}\sum_{n=1}^{+\infty}\int_{X} u_0(y)dm_x^{\ast n}(y)\frac{t^{n-1}}{(n-1)!}.$$
Therefore, to prove \eqref{d1704}, we only need to show that
$$
e^{-t}\sum_{n=1}^{+\infty}\int_{X} u_0(y)dm_x^{\ast n}(y)\frac{t^{n-1}}{(n-1)!}=\int_X u(z,t) dm_x(z).
$$
Now, by induction it is easy to see that
$$\int_{X} u_0(y)dm_x^{\ast n}(y) = \int_X \left(\int_{X} u_0(y)dm^{\ast (n-1)}_z(y)\right) dm_x(z).$$
Thus,
$$e^{-t}\sum_{n=1}^{+\infty}\int_X u_0(y)dm_x^{\ast n}(y)\frac{t^{n-1}}{(n-1)!}=e^{-t}\sum_{n=1}^{+\infty}\int_X \left(\int_{X} u_0(y)dm^{\ast (n-1)}_z(y)\right) dm_x(z)\frac{t^{n-1}}{(n-1)!} $$
$$= \int_{z\in X} \left(e^{-t}\sum_{n=1}^{+\infty} \int_{X} u_0(y)dm^{\ast (n-1)}_z(y) \frac{t^{n-1}}{(n-1)!}\right)dm_x(z)= \int_X u(z,t) dm_x(z),$$
where we have interchanged the series and integral applying the dominated convergence Theorem  because $$\left|e^{-t}\sum_{n=1}^{k} \int_{X} u_0(y)dm^{\ast (n-1)}_z(y) \frac{t^{n-1}}{(n-1)!}\right|\leq e^{-t}\sum_{n=1}^{+\infty} \int_{X} |u_0(y)|dm^{\ast (n-1)}_z(y) \frac{t^{n-1}}{(n-1)!}=:F(z) $$
and $F$ belongs to $L^1(X,\nu)$, thus to $L^1(X,m_x)$ for $\nu$-a.e. $x\in X$.
\end{proof}

\subsection{Infinite Speed of Propagation and Ergodicity}\label{infspeed}

In this section we study the  infinite speed of propagation of the heat flow $(e^{t\Delta_{m}})_{t \geq 0}$, that is, if it holds that
$$e^{t\Delta_{m}} u_0> 0 \ \ \ \hbox{for all} \ t > 0 \quad \hbox{whenever} \ \ 0 \leq u_0 \in L^2(X, \nu),\ u_0\not\equiv 0.$$
 We will see that this property is equivalent to a connectedness property of the space, to the ergodicity of the $m$-Laplacian $\Delta_{m}$ and  to the ergodicity of the measure $\nu$.

  Let  $[X, d, m]$ be a metric random walk space with invariant  measure $\nu$.  For a $\nu$-measurable set~$D$,  we set
  $$  N^m_D =\{x\in X   \,:\, m_x^{\ast n}(D)=0 , \ \forall n\in \N\}. $$
  For $n \in \N$, we also define     $$H^m_{D,n}=\{ x\in X\, :\, m_x^{\ast n}(D)>0\},$$
and
$$H^m_D :=
\bigcup_{n \in \N} H^m_{D,n} = \Big\{ x\in X\, :\, m_x^{\ast n}(D)>0\ \hbox{for some } n\in\N\Big\}.$$
Note that $N^m_D$ and $H^m_D$ are disjoint and   $$X=N^m_D\cup H^m_D.$$
 Observe also that $N_D^m,\ H_{D,n}^m$ and $H_D^m$ are $\nu$-measurable.

\begin{proposition}\label{propprev} Let  $[X, d, m]$ be a metric random walk space with invariant measure $\nu$.
For a $\nu$-measurable set~$D$, if $ N^m_D\neq\emptyset$ then:
\item{
1.} $$
\begin{array}{l}\displaystyle
m_x^{\ast n}(H^m_D)=0\quad \hbox{for every $x\in N^m_D$ and $n\in\N$},
\\ \\
\displaystyle m_x^{\ast n}(N^m_D)=1\quad \hbox{for every $x\in N^m_D$ and $n\in\N$}.
\end{array}$$

\item{ 2.}  If $\nu(X)< +\infty$ or $\nu$ is reversible, then \begin{equation}\label{good1cuatro}
\begin{array}{c}\displaystyle m_x^{\ast n}(H^m_D)=1\quad\hbox{for $\nu$-almost every $x\in H^m_D$, and for all $n\in\N$.}
\\ \\
\displaystyle m_x^{\ast n}(N^m_D)=0 \quad\hbox{for $\nu$-almost every $x\in H^m_D$, and for all $n\in\N$.}
\end{array}  \end{equation}

Consequently, for every $x\in N_D^m$ and $\nu$-a.e. $y\in H^m_D$ we have $m_x\bot m_y$, i.e. $m_x$ and $m_y$ are mutually singular.
\end{proposition}

\begin{proof} {\it 1}: Suppose that $m_x^{\ast k}(H^m_D)>0$ for some $x\in N^m_D$ and $k\in\N$, then, since $H^m_D=\cup_n H^m_{D,n}$ there exists $n\in \N$ such that $m_x^{\ast k}(H^m_{D,n})>0$ but in that case we have
$$m^{\ast (n+k)}_x(D)=\int_{z\in X} m_z^{\ast n}(D)dm_x^{\ast k}(z)\geq \int_{z\in H^m_{D,n}} m_z^{\ast n}(D)dm_x^{\ast k}(z)>0$$
since $m_z^{\ast n}(D)>0$ for ever $z\in H^m_{D,n}$, and this contradicts that $x\in N^m_D$. The second statement in 1. is then immediate.

{\it 2}: Fix  $n\in\N$.   Using statement 1. we have that  for any finite $\nu$-measurable set $A$,
\begin{equation}\label{sab1925} \nu(A\cap H_D^m)=\int_X m^{\ast n}_x(A\cap H_D^m) d\nu(x)=\int_{H_D^m} m^{\ast n}_x(A\cap H_D^m) d\nu(x)\end{equation}
because $m_x^{\ast n}(H_D^m)=0$ for every $x\in N_D^m$.

If $\nu(H_D^m)$ is finite, by \eqref{sab1925},
$$\nu(H^m_D)= \int_{H^m_D} m_x^{\ast n}(H^m_D)d\nu(x)\quad\forall n\in\N,$$
and, therefore,
$$m_x^{\ast n}(H^m_D)=1\quad\hbox{for $\nu$-almost every $x\in H^m_D$, and for all $n\in\N$.}
$$

On the other hand,  if $\nu(H_D^m)$   is not finite,  since the space is $\sigma$-finite,
we have that $H_D^m=\displaystyle \bigcup_{j=1}^{+\infty}{H^m_D\cap B_j}$, with $B_j$ open and $0<\nu(B_j)<+\infty$. Now,
by \eqref{sab1925} and  using reversibility of~$\nu$,
$$\nu(H^m_D\cap B_j)  =\int_{H_D^m} m^{\ast n}_x(H^m_D\cap B_j) d\nu(x)
=\int_{H_m^D\cap B_j} m^{\ast n}_x(H_D^m) d\nu(x);$$
thus
$$m_x^{\ast n}(H^m_D)=1\quad\hbox{for $\nu$-almost every $x\in H_m^D\cap B_j$. }
$$
Consequently,
$$m_x^{\ast n}(H^m_D)=1\quad\hbox{for $\nu$-almost every $x\in H^m_D$, and for all $n\in\N$.}
$$

The second statement in {\it 2\,} then follows.
\end{proof}

\begin{proposition}\label{proppart}  Let  $[X, d, m]$ be a metric random walk space with invariant   measure $\nu$ such that $\nu(X)< +\infty$ or $\nu$  is reversible. For a $\nu$-measurable set~$D$, we have that, for  every $n\in \mathbb{N}$ and for any finite $\nu$-measurable set $A$,
$$ \nu(A\cap H_D^m)=\int_{H_D^m} m^{\ast n}_x(A) d\nu(x),$$
and
$$ \nu(A\cap N_D^m)=\int_{N_D^m} m^{\ast n}_x(A) d\nu(x).$$
\end{proposition}

\begin{proof}  If $N_D^m=\emptyset$ the result follows trivially, so let us suppose that $N_D^m\neq\emptyset$.
 By~\eqref{sab1925}, and using Proposition~\ref{propprev}, we have that,
  for any  finite $\nu$-measurable set $A$,
  $$\nu(A\cap H_D^m)=\int_{H_D^m} m^{\ast n}_x(A\cap H_D^m) d\nu(x)=\int_{H_D^m} m^{\ast n}_x(A) d\nu(x)$$
since
$$m_x^{\ast n}(A)=m_x^{\ast n}(A\cap H_D^m)+m_x^{\ast n}(A\cap N_D^m)=m_x^{\ast n}(A\cap H_D^m)$$
for $\nu$-a.e. $x\in H_D^m$ and every $n\in \N$. Similarly, one proves the other statement.
\end{proof}

We have the following corollary.

\begin{corollary}\label{sab1936}  Let  $[X, d, m]$ be a metric random walk space with invariant   measure $\nu$ such that $\nu(X)< +\infty$ or $\nu$  is reversible.
For any   $\nu$-measurable set $D$, we have that $$\nu(N_D^m\cap D)=0.$$
Consequently, if $\nu(D)>0$, then $D\subset H_D^m$ up to a $\nu$-null set; therefore, for $\nu$-a.e. $x\in D$ there exists $n=n(x)\in \N$ such that $m^{\ast n}_x(D)>0$.
\end{corollary}

\begin{proof}
  If $D$ is finite, from Proposition~\ref{proppart},
$$\nu(N_D^m\cap D)=\int_{N_D^m} m^{\ast n}_x(D) d\nu(x)=0.$$

  If $D$ is not finite, then  $D=\displaystyle \bigcup_{j=1}^{+\infty}{D\cap B_j}$, with $B_j$ open and $0<\nu(B_j)<+\infty$. Then,
$$\nu(N_D^m\cap D)\leq\sum_{j=1}^{+\infty}\nu(N_D^m\cap D\cap B_j).$$
Now, from Proposition~\ref{proppart},  $\displaystyle \nu(N_D^m\cap D\cap B_j)=\int_{N_D^m}m_x^{*n}( D\cap B_j)d\nu(x)\le \int_{N_D^m}m_x^{*n}(D)d\nu(x)=0$, thus
\[\nu(N_D^m\cap D)=0.\qedhere\]
 \end{proof}

\begin{definition}\label{conexo01}{\rm
 A metric random walk space $[X, d, m]$ with invariant  measure $\nu$ is called {\it random-walk-connected} or {\it $m$-connected} if,
for any $D \subset X$  with $0<\nu(D)<+\infty$, we have that~$\nu(N_D^m)=0.$

A metric random walk space $[X, d, m]$ with invariant   measure $\nu$ is called  {\it weakly-$m$-connected} if,
for any open set $D \subset X$ with $0<\nu(D)<+\infty$, we have that
   $\nu(N_D^m)=0.$
   }
  \end{definition}

\begin{theorem}\label{despf01}  Let  $[X, d, m]$ be a metric random walk space with invariant  and reversible measure $\nu$.
  \item{  1.}
  The space is $m$-connected    if, and only if,  for any non-null $0 \leq u_0 \in L^2(X, \nu)$, we have $e^{t\Delta_{m}} u_0> 0$ $\nu$-a.e. for all $t > 0$.
\item{ 2.}  The space is weakly-$m$-connected if, and only if, for any non-null $0 \leq u_0 \in L^2(X, \nu)\cap C(X)$, we have $e^{t\Delta_{m}} u_0> 0$ $\nu$-a.e. for all $t > 0$.
\end{theorem}

 \begin{proof}   ({\it 1}, $\Rightarrow$): Given a non-null $0 \leq u_0 \in L^2(X, \nu)$, there exist $D \subset X$ with $0<\nu(D)<+\infty$ and $\alpha > 0$, such that $u_0 \geq \alpha \1_D$. Therefore, by Theorem~\ref{expannsion1},
 $$ e^{t\Delta_{m}} u_0(x) \geq \alpha e^{t\Delta_{m}} \1_D(x) =  \alpha e^{-t}\sum_{n=0}^{\infty} m_x^{\ast n}(D)\frac{t^n}{n!} > 0\quad\hbox{for $\nu$-a.e. $x\in X$}\, .$$
 Indeed, if $x\in X\setminus N_D^m$ we have that $x\in H_D^m$, so there exists $n\in \mathbb{N}$ such that $ m_x^{\ast n}(D) > 0$. Then, since $\nu(N_D^m)=0$, we conclude.

({\it 1}, $\Leftarrow$):  Take $D
\subset X$   with $0<\nu(D)<+\infty$, we have that
$$ e^{t\Delta_{m}} \1_D(x) =  e^{-t}\sum_{n=0}^{\infty} m_x^{\ast n}(D)\frac{t^n}{n!} > 0\quad\hbox{for $\nu$-a.e. $x\in X$}.
 $$
Moreover, since $m_x^{\ast 0}=\delta_x$ and $ m_x^{\ast n}(D)=0$ for every $x\in N_D^m$ and $n\ge 1$, we get
$$   e^{-t} \sum_{n=0}^{\infty} m_x^{\ast n}(D)\frac{t^n}{n!} =    e^{-t}  \1_{D}(x)\quad\hbox{for   $x\in N_D^m$},
$$ thus
$$   \1_{D}(x) > 0\quad\hbox{for $\nu$-a.e. $x\in N_D^m$}.$$
 Hence, by Corollary~\ref{sab1936}, $\nu(N_D^m)=0$.

The proof of  ({\it 2}, $\Rightarrow$) is similar.  ({\it 2}, $\Leftarrow$): Take $D\subset X$  open  with $0<\nu(D)<+\infty$, since  $\nu$ is regular there exists a compact set $K\subset D$ with $ \nu(K)>0$. By Urysohn's lemma we may find a continuous function $0\leq u_0\leq 1$ such that $u_0=0$ on $X\setminus D$ and $u_0=1$ on $K$, thus $u_0\leq \1_{D}$. Hence
$$e^{-t}\sum_{n=0}^{\infty}m_x^{*n}(D)\frac{t^n}{n!}\geq e^{t\Delta_{m}}u_0(x)>0\quad\hbox{for $\nu$-a.e. $x\in X$}.$$
So we conclude as before.
 \end{proof}

\begin{remark}\label{sax01}{\rm  In the preceding proof, when $N_D^m=\emptyset$, we obtain that in fact
$$ e^{t\Delta_{m}} u_0(x) > 0\quad\hbox{for all $x\in X$ and for all $t>0$}\, .
$$
We will say that the metric random walk space is {\it  strong $m$-connected} when this happens  for any non-$\nu$-null $0 \leq u_0 \in L^2(X, \nu)$,  and {\it weakly strong $m$-connected} if it holds for any non-null $0 \leq u_0 \in L^2(X, \nu)\cap C(X)$.
}\end{remark}

 The following result gives a characterization of $m$-connectedness in terms of the $m$-interaction of sets.

\begin{proposition}\label{extra001}
 Let $[X,d,m]$ be a metric random walk space with  reversible  measure $\nu$.
The following statements are equivalent:
\item{ 1.} $X$ is $m$-connected.
\item {2.} If $ A,B\subset X$ are $\nu$-measurable   non-$\nu$-null sets such that
$A\cup B=X$, then
$L_m(A,B)> 0$.
\end{proposition}

\begin{proof}
${\it 1}\Rightarrow {\it 2}$: Assume that $X$ is $m$-connected and let $A,B$ be as in statement 2. If
$$0=L_m(A,B)=\int_A\int_Bdm_x(y)d\nu(x),$$
 then
\begin{equation}\label{tehabo001}
\begin{array}{l}
\displaystyle m_x(B)=0\ \hbox{ for all } x\in A\setminus N_1,\ \nu(N_1)=0.
\end{array}
\end{equation}

 Now, since $\nu$ is invariant for $m$, $$0=\nu(N_1)=\int_Xm_x(N_1)d\nu(x),$$
and, consequently, there exists $N_2\subset X$, $\nu(N_2)=0$, such that
$$m_x(N_1)=0\quad\forall x\in X\setminus N_2.$$
Hence, for $x\in A\setminus (N_1\cup N_2)$,
$$\begin{array}{l} \displaystyle
m_x^{*2}(B)=\int_X\1_B(y)dm_x^{*2}(y)=\int_X\left(\int_X\1_B(y)dm_z(y)\right)dm_x(z)
\\ \\
\displaystyle\phantom{m_x^{*2}(B)}
=\int_X m_z(B)dm_x(z)=\int_A m_z(B)dm_x(z)+\underbrace{\int_B m_z(B)dm_x(z)}_{=0,\,\hbox{\tiny since $x\in A\setminus N_1$}}
\\ \\
\displaystyle\phantom{m_x^{*2}(B)}
=\underbrace{\int_{A\setminus N_1} m_z(B)dm_x(z)}_{=0,\,\hbox{\tiny since $z\in A\setminus N_1$}}+\underbrace{\int_{N_1} m_z(B)dm_x(z)}_ {=0,\,\hbox{\tiny since $x\notin N_2$}}=0
\end{array}$$

 Working as  above, we find $N_3\subset X$, $\nu(N_3)=0$, such that
$$m_x(N_1\cup N_2)=0\quad\forall x\in X\setminus N_3.$$
Hence, for $x\in A\setminus (N_1\cup N_2\cup N_3)$,  we have that
$$\begin{array}{l} \displaystyle
m_x^{*3}(B)=\int_X\1_B(y)dm_x^{*3}(y)=\int_X\left(\int_X\1_B(y)dm_z^{*2}(y)\right)dm_x(z)
\\ \\
\displaystyle\phantom{m_x^{*3}(B)}
=\int_X m_z^{*2}(B)dm_x(z)
\le \int_A m_z^{*2}(B)dm_x(z)+\underbrace{\int_B m_z^{*2}(B)dm_x(z)}_{=0,\,\hbox{\tiny since $x\in A\setminus (N_1\cup N_2)$}}
\\ \\
\displaystyle\phantom{m_x^{*3}(B)}
\le \underbrace{\int_{A\setminus (N_1\cup N_2)} m_z^{*2}(B)dm_x(z)}_{=0,\,\hbox{\tiny since $z\in A\setminus  (N_1\cup N_2)\ $}}+\underbrace{\int_{N_1\cup N_2} m_z^{*2}(B)dm_x(z)}_ {=0,\,\hbox{\tiny since $x\notin N_3$}}=0.
\end{array}$$

Inductively,
we obtain that
$$m_x^{*n}(B)=0\quad\hbox{for $\nu$-a.e $x\in A$ and every $n\in\mathbb{N}$}.$$
Consequently,
$$A\subset N_B^m$$ up to a $\nu$-null set,
which is a contradiction.

${\it 2}\Rightarrow {\it 1}$: Assume that statement 2 holds. If $X$ is not $m$-connected, then there exists a $\nu$-measurable set $D\subset X$, $0<\nu(D)<+\infty$, such that  $\nu(N_D^m)>0$.  Moreover, by Corollary \ref{sab1936}, $D \subset H^m_D$. Hence,
  we have that
$$\nu(N_D^m)>0, \ \nu(H_D^m)>0
\ \hbox{and} \ X= N_D^m\cup H_D^m,$$
thus, by the hypothesis, $L_m( N_D^m, H_D^m) >0$,
which is a contradiction.
\end{proof}

 Observe that the metric random walk space given in Example~\ref{JJ}~\eqref{dom001} is $m$-connected. This space  has Ollivier-Ricci curvature equal to zero. In the next result we see that  metric random walk spaces with positive Ollivier-Ricci curvature are  $m$-connected.  We will also see that connected graphs are always $m$-connected in Theorem~\ref{graphscurvpos}.

\begin{theorem}\label{1despf01} Let $[X, d, m]$ be a metric random walk space with finite invariant measure $\nu$. Assume that the Ollivier-Ricci curvature $\kappa$ satisfies $\kappa >0$.
Then,  $[X, d, m]$ with $\nu$ is $m$-connected and   weakly  strong $m$-connected.
\end{theorem}

\begin{proof}
Under the hypothesis $\kappa >-\infty$, recall that $\kappa\le 1$ by definition, Y.~Ollivier in~\cite[Proposition~20]{O} proves the following $W_1$ contraction property:

 {\it Let $[X,d,m]$ be a metric random walk space. Then, for any two probability distributions, $\mu$ and $\mu'$,}
\begin{equation}\label{oli1prevprop}
    W^d_1(\mu\ast m^{\ast n}, \mu'\ast m^{\ast n}) \leq (1 - \kappa)^n   W^d_1(\mu, \mu').
\end{equation}
 Hence, under the hypothesis $\kappa >0$, Y.~Ollivier in~\cite[Corollary~21]{O} proves that the invariant measure $\nu$ (exists and) is unique up to a multiplicative constant,  and that, for $\nu$ such that $\nu(X)=1$, the following hold:
\begin{equation}\label{oli1prev}
  \begin{array}{l}
 (i)\qquad W^d_1(\mu\ast m^{\ast n}, \nu) \leq (1 - \kappa)^n  W^d_1(\mu, \nu) \quad \forall n \in \N, \  \forall\mu\in \mathcal{P}(X),\\ \\
 (ii)\qquad W^d_1( m^{\ast n}_x, \nu) \leq \displaystyle(1 - \kappa)^n  \frac{W^d_1(\delta_x, m_x)}{\kappa} \quad \forall n \in \N, \  \forall x\in X.
 \end{array}
\end{equation}
 So we will suppose, without loss of generality, that $\nu(X)=1$.
 By~\eqref{oli1prev} and \cite[Theorem 6.9]{Villani2}, we have that
\begin{equation}\label{oli1prevbist}\begin{array}{c}\displaystyle \mu\ast m^{\ast n} \rightharpoonup \nu \quad \hbox{weakly as measures, } \forall\mu\in \mathcal{P}(X),\\ \\
\displaystyle  m^{\ast n}_x \rightharpoonup \nu \quad \hbox{weakly as measures, for every $x\in X$}.
\end{array} \end{equation}

 Let us now see that the space is $m$-connected if $\kappa>0$.
Take $D \subset X$  with $0<\nu(D)<+\infty$ and suppose that $\nu(N_D^m)>0$. By Proposition \ref{sab1936}, we have  $\nu (H_D^m)>0$.
 Let
$$\mu:=\frac{1}{\nu(H_D^m)}\nu\res H_D^m\in \mathcal{P}(X) ,$$
and
$$\mu':=\frac{1}{\nu(N_D^m)}\nu\res N_D^m\in \mathcal{P}(X) ,$$
 Now, by Proposition~\ref{proppart},
$$ \mu\ast m^{\ast n}=\mu,$$
and
$$ \mu'\ast m^{\ast n}=\mu',$$
but then, by  \eqref{oli1prevprop}, we get
$$ W_1(\mu, \mu')=W_1(\mu\ast m^{\ast n}, \mu'\ast m^{\ast n}) \leq (1 - \kappa)^n   W_1(\mu, \mu')$$
which is only possible if $W_1(\mu, \mu')=0$ since $1-\kappa<1$. Hence,
$$\mu=\mu',$$
and this implies $1=\mu'(N_D^m)=\mu(N_D^m)=0$ which gives a contradiction. Therefore, $\nu(N_D^m)=0$ so  the space is $m$-connected.

 If $D$ is open and $\nu(D)>0$ then $N_D^m=\emptyset$. Indeed, for $x\in N_D^m$, by \eqref{oli1prevbist}, we have
\[0<\nu(D)\le \liminf_{n}m^{\ast n}_x(D)=0.  \qedhere\]
\end{proof}

\begin{remark}{\rm   In~\cite{HLL}, it is  shown that uniqueness of the invariant probability measure implies its ergodicity. Consequently, Theorem \ref{1despf01} follows from \cite[Corollary~21]{O} (see Theorem \ref{ergconect} for the relation between ergodicity and $m$-connectedness). We have presented the result for the sake of completeness and using the framework of $m$-connectedness.

Note that, in the previous result, a condition involving the random walk and the metric on the ambient space yields the $m$-connectedness of the metric random walk space, which is, a priori, unrelated to the metric.
}
\end{remark}

 \begin{proposition} Let     $[X, d, m]$ be a metric random walk space with invariant   measure $\nu$ such that $\hbox{supp}\, \nu=X$ and either $\nu(X)< +\infty$, or $\nu$ is reversible. Suppose further that $[X,d,m]$  has the strong-Feller property and $(X,d)$ is connected, then    $[X, d, m]$ with $\nu$ is strong $m$-connected.
\end{proposition}
\begin{proof}
 Recall that, since $[X,d,m]$ has the strong-Feller property,   $[X,d,m^{\ast k}]$ also has this property for any $k\in \mathbb{N}$.  Take  $D$ with $0<\nu(D)<+\infty$. Let us see first that $H_D^m$ is open or, equivalently, that $N_D^m$ is closed; indeed, if we have $(x_n)\subset N_D^m$ such that $\lim_{n\to \infty}x_n=x\in X$
then
$$m_x^{\ast k}(D)= \lim_{n\to\infty}m_{x_n}^{\ast k}(D)=0  $$
for any $k\in \N$, thus $x\in N_D^m$.

On the other hand, if $ m_x(H_D^m)<1$ for some $x\in H_D^m$, since $[X,d,m]$ has the strong-Feller property,  there exists $r>0$ such that $m_y(H_D^m)<1$ for every $y\in B_r(x)\subset H_D^m$. Therefore, by \eqref{good1cuatro}, $\nu(B_r(x))=0$, which is a contradiction since $\hbox{supp}\, \nu=X$. Hence,
$$ m_x(H_D^m)=1\quad\hbox{if, and only if,}\quad x\in H_D^m.
$$
 Then, given $(x_n)\subset H_D^m$ such that $\lim_{n\to \infty}x_n=x\in X$, we have
$$m_x(H_D^m)=\lim_{n\to\infty}m_{x_n}(H_D^m)=1,$$
so $x\in H_D^m$. Therefore, $H_D^m$ is also closed and then, since  $X$ is connected,  we have that $X=H_D^m$, which implies that  $N_D^m=\emptyset$.
\end{proof}

  \begin{theorem}\label{graphscurvpos}
  Let $[V(G), d_G, (m^G_x)]$ be the metric random walk space associated with the locally finite weighted   connected   graph $G = (V(G), E(G))$.     Then  $[V(G),d_G,m^G]$  with $\nu_G$ is   strong $m$-connected.
 \end{theorem}

 \begin{proof} Take $D \subset V(G)$ with $\nu_G(D) >  0$, and let us see that $N^{m^G}_{D}=\emptyset.$ Suppose that there exists $y \in N^{m^G}_{D}$, this implies that
 \begin{equation}\label{blanca1}
 (m^G)_{y}^{*n}(D)= 0 \quad \forall \, n \in \N.
 \end{equation}
  Now, given $x \in D$, there is a path of length $m$, $\{x, z_1, z_2,\dots, z_{m-1}, y\}$, $x\sim z_1\sim z_2\sim\dots\sim z_{m-1}\sim y$, and then
$$(m^G)_{y}^{*n}(\{x\})\ge \frac{w_{yz_{m-1}}w_{z_{m-1}z_{m-2}}\cdots w_{z_{2}z_{1}}w_{z_{1}x}}{d_yd_{z_{m-1}}d_{z_{m-2}}\cdots d_{z_2}d_{z_1}}>0,$$
which is in contradiction with \eqref{blanca1}.
\end{proof}

The next examples show that there is no relation between $m$-connectedness and classical connectedness, i.e., there are connected metric random walk spaces that are not $m$-connected and, conversely, there are $m$-connected metric random walk spaces that are not connected.

 \begin{example}\label{ejm01}{\rm  Take $([0,1],d)$ with $d$ the Euclidean distance and let $C$ be the Cantor set. Let $\mu$ be the Cantor distribution, that is, the probability measure whose cumulative distribution function $F(x) = \mu([0,x))$ is the Cantor function. We have that $\mu$ is singular with respect to the Lebesgue measure and its support is the Cantor set. We denote $\eta := \mathcal{L}^1 \res [0,1]$ and define the random walk
$$m_x:= \left\{ \begin{array}{ll} \eta \quad &\hbox{if} \ \ x \in [0,1] \setminus C , \\ \\ \mu \quad &\hbox{if} \ \ x \in C . \end{array} \right.$$
Then, $\nu=\eta+ \mu$ is invariant and reversible. Indeed,
$$\int_{(0,1)}\int_{(0,1)}f(y)dm_x(y)d\nu(x)=\int_{(0,1)\setminus C}\int_{(0,1)}f(y)dy dx+\int_{C}\int_{C}f(y)d\mu(y)d\mu(x)=$$
$$=\int_{(0,1)}f(y)dy+\int_{C}f(y)d\mu(y)=\int_{(0,1)}f(y)d\nu(y),$$
Similarly, we prove that $\nu$ is reversible.

On the other hand,   $m_x^{\ast n}=m_x$ for any $x\in (0,1)$ and $n\in\N$. In fact, if $x\in C$, we have
$$\int_{y\in X}f(y)dm_x^{*2}(y)=\int_{z \in (0,1)}\left(\int_{y\in 0,1)}f(y)dm_z(y)\right)dm_x(z) = \int_{z \in (0,1)}\left(\int_{y\in (0,1)}f(y)dm_z(y)\right)d\mu(z)$$ $$= \int_{z \in C}\left(\int_{y\in (0,1)}f(y)dm_z(y)\right)d\mu(z) = \int_{z \in C}\left(\int_{y\in C}f(y)d\mu (y)\right)d\mu(z) = \int_X f(y) dm_x(y),$$
and the proof for $x \in (0,1) \setminus C$ is similar.

Consequently,  $m_x^{\ast n}(C)=0$ for every $x\in(0,1)\setminus C$ and for every $n\in \N$, so that $\nu(N_C^m)\geq \nu((0,1)\setminus C)=1>0$ and, therefore, the space $[(0,1), d, m]$ is not $m$-connected.

 For $x,y\in (0,1)$, $x\neq y$, if $x,y\in C$, or $x,y\in (0,1)\setminus C$, then  $W_1(m_x,m_y)=0$ and hence $\kappa(x,y)=1$; otherwise, if $x\in C$ and $y\in (0,1)\setminus C$ then $W_1(m_x,m_y)=W_1(\mu,\eta)$.
Hence $\kappa(x,y)=1-\frac{W_1(\mu,\eta)}{|x-y|}$. Consequently, since $\displaystyle \kappa=\inf_{x\neq y}\kappa(x,y)$ and $\kappa(x,y)=1$ if $x,y\in C$ or $x,y\in(0,1)\setminus C$, we get
$$\kappa=\inf_{x\in C,y \notin C}\left(1-\frac{W_1(\mu,\eta)}{|x-y|}\right)=-\infty \,.$$}
\end{example}

\begin{example}\label{megusta1}{\rm Let $\Omega:=\Big(\big]-\infty,0\big]\cup \big[\frac12,+\infty\big[\Big)\times \mathbb{R}^{N-1}$ and consider the metric  random walk space $[\Omega,d,m^{J,\Omega}]$, with $d$ the Euclidean distance and $J(x) = \frac{1}{\vert B_ 1(0) \vert} \1_{B_1(0)}$ (see Example~\eqref{JJ}~\eqref{dom001}). It is easy to see that  this space with reversible and invariant measure $\nu=\mathcal{L} \res \Omega$ is $m$-connected, but $(\Omega,d)$ is not connected.
Let us see  that its    Ollivier-Ricci curvature $\kappa$ is negative. Indeed, take  $x = (-\frac{1}{2},0, \ldots, 0)\in \Omega$,  and  $y= (2,0, \ldots, 0)\in \Omega$. Then, we have that $u(x_1,x_2, \ldots x_N)=  -x_1$  is a Kantorovich potential   for the transport of $m^{J}_{x}$ to $m^{J}_{y}$, consequently
$$\begin{array}{l}\displaystyle
W_1(m^{J,\Omega}_{x},m^{J,\Omega}_{y})\ge \int_\Omega u(z)\left(d m^{J,\Omega}_{x}(z)-d m^{J,\Omega}_{y}(z)\right)
\\ \\ \displaystyle
\phantom{W_1(m^{J,\Omega}_{x},m^{J,\Omega}_{y})}
= \int_\Omega u(z)\left(d m^{J}_{x}(z)-d m^{J}_{y}(z)\right) + \left(\int_{\R^N \setminus \Omega} d m^{J}_{x}(z) \right)u(x)
\\ \\ \displaystyle
\phantom{W_1(m^{J,\Omega}_{x},m^{J,\Omega}_{y})}
=\int_{\R^N} u(z)\left(d m^{J}_{x}(z)-d m^{J}_{y}(z)\right) +\int_{\R^N \setminus \Omega} (u(x) - u(z)) d m^{J}_{x}(z)  \\ \\
\displaystyle
\phantom{W_1(m^{J,\Omega}_{x},m^{J,\Omega}_{y})}
= W_1(m^{J}_{x},m^{J}_{y}) +  \int_{\R^N \setminus \Omega} \left( \frac12 + z_1\right) d m^{J}_{x}(z)
\\  \\
\displaystyle
\phantom{W_1(m^{J,\Omega}_{x},m^{J,\Omega}_{y})}
> d(x,y)\,.
\end{array}
$$
Therefore, the Ollivier-Ricci curvature of $[\Omega,d,m^{J,\Omega}]$
$$\kappa \leq \kappa(x,y) = 1 - \frac{W_1(m^{J,\Omega}_{x},m^{J,\Omega}_{y})}{d(x,y)} < 0.$$

For $\Omega=\Big(\big]-\infty,0\big]\cup \big[2,+\infty\big[\Big)\times \mathbb{R}^{N-1}$, neither $[\Omega,d,m^{J,\Omega}]$ with $\nu=\mathcal{L} \res \Omega$ is $m$-connected, nor  $(\Omega,d)$ is   connected.  As above, we can prove  that its    Ollivier-Ricci curvature is negative.

 In a similar way we have that, for $\Omega=\mathbb{R}^N\setminus (0,2)^N$, $[\Omega,d,m^{J,\Omega}]$ with $\nu=\mathcal{L} \res \Omega$ is   $m$-connected,    $(\Omega,d)$ is   connected and its    Ollivier-Ricci curvature is $\kappa<0$.
}
\end{example}

  We will now relate the $m$-connectedness property with other known concepts in the literature. Let us begin with the concept of ergodicity (see, for example, \cite{HLL}).

\begin{definition}{\rm Let $[X, d, m]$ be a metric random walk space with invariant  probability measure~$\nu$. A Borel set $B \subset X$ is said to be {\it invariant} with respect to the random walk $m$ if $m_x(B) = 1$ whenever $x$ is in $B$.

 The invariant  probability measure~$\nu$ is said to be {\it ergodic} if $\nu(B) = 0$ or $\nu(B) = 1$ for every invariant set $B$ with respect to the random walk $m$.
}
\end{definition}

\begin{theorem}\label{ergconect} Let $[X, d, m]$ be a metric random walk space with invariant
 probability measure~$\nu$. Then, the following assertions are equivalent:
\begin{itemize}
\item[(i)] $[X, d, m]$ with $\nu$ is $m$-connected.\\[-6pt]

\item[(ii)] $\nu$ is ergodic.
\end{itemize}
\end{theorem}

\begin{proof} $(i) \Rightarrow (ii)$. If $\nu$ is not ergodic, there exists an invariant set $B$ with respect to the random walk $m$ such that $0 < \nu(B) < 1$. However, note that $B$ is also invariant with respect to $m^{\ast 2}$. Indeed,
$$m_x^{\ast 2}(B)=\int_X m_z(B) dm_x(z)\ge \int_B m_z(B)dm_x(z)=m_x(B)=1$$
for every $x\in B$. Inductively, we obtain that, in fact, $B$ is invariant for $m^{\ast n}$, for every $n\in\N$. Therefore, since
$$\nu(B) = \int_X m^{\ast n}_x(B) d\nu(x) \geq \int_B m^{\ast n}_x(B) d\nu(x) = \nu(B) \quad \hbox{for every $n\in\N$},$$
we obtain that $m^{\ast n}_x(B) = 0$ for every $n\in \N$ and $\nu$-a.e. $x \in X \setminus B$.
 Therefore, $X \setminus B \subset N_B^m$ $\nu$-a.e., thus $\nu( N_B^m) >0$ and, consequently,  $[X, d, m]$ with $\nu$ is not $m$-connected.

\noindent $(ii) \Rightarrow (i)$.  Let $D \subset X$ be a $\nu$-measurable set with $\nu(D) >0$. By Proposition \ref{propprev} we have that $N_D^m$ is invariant with respect to the random walk $m$. Then, since $\nu$ is ergodic, we have that $\nu(N_D^m) = 0$ or $\nu(N_D^m) = 1$. Now, since $\nu(D) >0$, by Corollary \ref{sab1936}, we have that $\nu(N_D^m) = 0$ and, consequently, $[X, d, m]$ with $\nu$ is $m$-connected.
\end{proof}

 Following Bakry, Gentil and Ledoux \cite{BGL}, we give the following definition.

\begin{definition} Let    $[X, d, m]$ be a metric random walk space with invariant   measure $\nu$. We say that $\Delta_m$ is {\it ergodic} if, for $u\in {\rm Dom}(\Delta_m)$, $\Delta_m u = 0$ implies  that $u$ is constant (being this constant $0$ if $\nu$ is not finite).
\end{definition}

\begin{theorem}\label{ergo1}
Let    $[X, d, m]$ be a metric random walk space with finite invariant   measure $\nu$. Then,
 $$  \Delta_m \ \hbox{  is ergodic} \ \Leftrightarrow \ [X,d,m] \ \hbox{is random walk connected.}$$
\end{theorem}

\begin{proof}
($\Rightarrow$): Suppose that $\nu(X) < +\infty$ and $[X,d,m]$ is not $m$-connected, then there exists $D \subset X$ with $\nu(D) >0$ such that $\nu(N^m_D)>0$. Recall that $\nu(H^m_D)>0$.  Consider the function $$u(x)=\1_{H_D^m}(x),$$
and note that $u\in L^2(X, \nu)$ since $\nu$ is finite.
Now,
$$\Delta_m u(x)= \int_X\big(\1_{H_D^m}(y)-\1_{H_D^m}(x)\big)dm_x(y)= m_x(H_D^m)-\1_{H_D^m}(x),
$$
hence,
 by Proposition \ref{propprev},
 $$\Delta_m u =0\quad\hbox{$\nu$-a.e.},$$
but $u$ is not equal to a constant $\nu$-a.e., and, consequently, $\Delta_m$ is not ergodic.

($\Leftarrow$): Suppose now that $[X,d,m]$ is $m$-connected and that there exists $u\in L^2(X, \nu)$ such that $\Delta_m u=0$ $\nu$-a.e. but $u$ is not $\nu$-a.e. equal to a constant function. Then, we may find $U$, $V\subset X$ with positive $\nu$-measure such that $u(x)<u(y)$ for every $x\in U$ and $y\in V$. Note that
\begin{equation}\begin{array}{l}\displaystyle \Delta_{m^{\ast n}}u(x)
=\int_X (u(y)-u(x))dm^{\ast n}_x(y)
\\ \\ \displaystyle \phantom{\Delta_{m^{\ast n}}u(x)}
=\int_X\int_X (u(z)-u(x))dm_y(z)dm^{\ast (n-1)}_x(y) \\ \\ \displaystyle
\phantom{\Delta_{m^{\ast n}}u(x)}
= \int_X\int_X (u(z)-u(y))dm_y(z)dm^{\ast (n-1)}_x(y)+\int_X\int_X (u(y)-u(x))dm_y(z)dm^{\ast (n-1)}_x(y) \\ \\ \displaystyle
\phantom{\Delta_{m^{\ast n}}u(x)}
=\int_X\Delta_m u(y)dm^{\ast (n-1)}_x(y)+\int_X (u(y)-u(x))dm^{\ast (n-1)}_x(y)\\ \\ \displaystyle
\phantom{\Delta_{m^{\ast n}}u(x)}
=\int_X\Delta_m u (y)dm^{\ast (n-1)}_x(y)+\Delta_{m^{\ast (n-1)}}u(x) \ ,
\end{array}\end{equation}
thus
\begin{equation}\label{deltainduc}
\left|\Delta_{m^{\ast n}}u(x)\right|\leq \int_X \left|\Delta_m u (y)\right|dm^{\ast (n-1)}_x(y)+\left|\Delta_{m^{\ast (n-1)}}u(x)\right|.
\end{equation}
Now, using the invariance of $\nu$,
$$\int_X\int_X \left|\Delta_m u (y)\right| dm^{\ast (n-1)}_x(y)d\nu(x)=\int_X\left|\Delta_m u (x)\right| d\nu(x)=0$$
so
$$\int_X\left|\Delta_m u (y)\right|dm^{\ast (n-1)}_x(y)=0 \quad \hbox{for} \ \nu\hbox{-a.e.} \ x\in X,$$
thus, by induction on \eqref{deltainduc}, $\Delta_{m^{\ast n}}u(x)=0$ for $\nu$-a.e. $x\in X$ and every $n\in \N$. Since $[X,d,m]$ is $m$-connected we have $\nu(N_V^m)=\nu(X\setminus H_V^m)=0$, so there exists $n\in \N$ such that $\nu(U\cap H_{V,n}^m)>0 $. Consequently, we get a contradiction:
\[0=\mathcal{H}_{m^{\ast n}}(u)=\int_X\int_X \nabla u(x,y)^2 dm^{\ast n}_x(y)d\nu(x)\geq \int_{U\cap H_{V,n}^m}\int_{V}\nabla u(x,y)^2 dm^{\ast n}_x(y)d\nu(x)>0 \ .\qedhere\]
\end{proof}

Let $[X, d, m]$ be a metric random walk space with invariant and reversible  measure $\nu$. It is easy to see that $\Delta_m$ is  ergodic if, and only if, $e^{t \Delta_m} f= f$ for all $t \geq 0$  implies that $f$ is constant. Moreover, we have the following result.

\begin{proposition}\label{asinttot} Let $[X, d, m]$ be a metric random walk space with invariant and reversible  measure $\nu$. For every $f \in L^2(X,\nu)$,
$$\lim_{t \to \infty} e^{t \Delta_m} f = f_\infty \in \{ u \in L^2(X,\nu) \, : \, \Delta_m u = 0 \}.$$
Suppose further that $\Delta_m$ is ergodic:
\begin{itemize}

\item[(i)] If $\nu(X) = +\infty$, then $ f_\infty = 0.$\\[-6pt]

\item[(ii)] If $\nu(X) < +\infty$, then $\displaystyle f_\infty = \frac{1}{\nu(X)} \int_X f(x) d \nu(x).$

\end{itemize}
\end{proposition}

\begin{proof} The first result follows from \cite[Theorem 3.11]{Brezis}. The second part is a consequence of the ergodicity of $\Delta_m$ and the conservation of mass \eqref{consermass}.
\end{proof}

  When the invariant measure is a probability measure, the relation between both concepts of ergodicity, the one for the invariant measure and the one for the Laplacian was known; see, for example,~\cite{HLL}. Let us now give another characterization of the ergodicity in terms of geometric properties.

\begin{lemma}\label{charactt} Let $[X, d, m]$ be a metric random walk space with invariant and reversible measure $\nu$ and assume that $\nu(X) < +\infty$. Then for every $\nu$-measurable subset $D \subset X$ we have
\begin{equation}\label{EgoodI}\Delta_m\1_D(x)=0\quad \Leftrightarrow\quad m_{x}(D)=\1_D(x),\end{equation}
and
\begin{equation}\label{EgoodII}\Delta_m\1_D=0\quad \Leftrightarrow\quad P_m(D)=0\quad\Leftrightarrow\quad \int_D \mathcal{H}_{\partial D}^m(x)d\nu(x)=-\nu(D).\end{equation}
\end{lemma}
\begin{proof} Since
$$\Delta_m\1_D(x)=\int_X\big(\1_D(y)-\1_D(x)\big)dm_x(y)=m_x(D)-\1_D(x)$$
we have that
$\Delta_m\1_D(x)=0$ if, and only if, $ \1_D(x)-m_x(D)=0$, and we get \eqref{EgoodI}.

Suppose now that  $\Delta_m\1_D=0$, then $\1_D(x)= m_x(D)$, thus integrating this expression over $D$ with respect to~$\nu$, we get
$$P_m(D)=\nu(D)-\int_D\int_D dm_x(y)d\nu(x)=0.$$ Conversely, if $\displaystyle P_m(D)=0,$ we have $$\displaystyle \nu(D)=\int_Dm_x(D) d\nu(x).$$ Then, on one hand, $$m_x(D)=\1_D(x)\quad\hbox{for $\nu$-a.e. $x\in D$,}$$ and, on the other hand, since
$$   \nu(D)=\int_Xm_x(D)d\nu(x)=\int_{D}m_x(D)d\nu(x)+\int_{X\setminus D}m_x(D)d\nu(x),$$
we get
$$
\int_{X\setminus D}m_x(D)d\nu(x)=0,$$
thus $$m_x(D)=\1_D(x)\quad\hbox{for $\nu$-a.e. $x\in X\setminus D$.}$$
Therefore,
$$m_x(D)=\1_D(x)\quad\hbox{for $\nu$-a.e. $x\in X$}$$
and, by  \eqref{EgoodI}, we get $\Delta_m \1_D =0$.

For the second equivalence, by \eqref{1secondf021}, we have that
$$\int_D \mathcal{H}^m_{\partial D}(x) d\nu(x)=2P_m(D) -\nu(D),$$
thus $P_m(D)=0$ if, and only if, $\displaystyle\int_D \mathcal{H}^m_{\partial D}(x) d\nu(x)= -\nu(D)$.
\end{proof}

\begin{theorem}\label{ERggB} Let $[X, d, m]$ be a metric random walk space with invariant and reversible measure $\nu$ and assume that $\nu(X)  < +\infty$.
The following  facts are equivalent:
\begin{enumerate}
\item\label{erg001}$\Delta_m$ is ergodic;\\[-6pt]
\item\label{erg002} $\Delta_m\1_D=0$ for a $\nu$-measurable set $D$ implies that $\1_D$ is constant;\\[-6pt]
\item\label{erg003} $P_m(D)>0$ for every $\nu$-measurable set $D$ such that $0<\nu(D)<\nu(X) $;\\[-6pt]
\item\label{erg004} The $\nu$-mean value of the $m$-mean curvature of $\partial D$ in $D$ satisfies $$\displaystyle\frac{1}{\nu(D)}\int_D \mathcal{H}_{\partial D}^m(x)d\nu(x)>-1\quad\hbox{for every $\nu$-measurable set $D$ such that $0<\nu(D)<\nu(X) $. }$$
\end{enumerate}
\end{theorem}

\begin{proof}
Obviously, ~\eqref{erg001} implies~\eqref{erg002}, and this yields \eqref{erg003} since $P_m(D)=0$ implies, by Lemma \ref{charactt}, that $\Delta_m\1_D=0$. Also, by Lemma \ref{charactt}, \eqref{erg003} implies~\eqref{erg002}. Let us now see that \eqref{erg002} implies~\eqref{erg001}: Suppose that  $\Delta_m$ is not ergodic, then the space is not $m$-connected and, consequently, there exists $D\subset X$ with $\nu(D)>0$ such that $0<\nu(N_D^m)<1$; but, from Proposition~\ref{propprev},
$$\Delta_m\1_{N_D^m}(x)=m_x(N_D^m)-\1_{N_D^m}(x)=0,$$
and this implies that $\1_{N_D^m}$ should be constant, which is a contradiction with $0<\nu(N_D^m)<\nu(X) $. The equivalence with~\eqref{erg004} is evident   by  the second equivalence in  Lemma \ref{charactt}.
\end{proof}

\section{Functional Inequalities}\label{sect03}

Let $[X, d, m]$ be a metric random walk space with invariant and reversible  measure $\nu$ such that $\nu(X) < +\infty$.   In this section we will further assume that $\nu(X)=1$, i.e. that $\nu$ is a probability measure. Note that we may always work with $\frac{1}{\nu(X)}\nu$.

\subsection{Spectral Gap and Poincar\'{e} Inequality}

We denote the mean value of $f \in L^1(X, \nu)$ (or the  expected value  of $f$) with respect to $\nu$ by $$\nu(f):= \mathbb{E}_\nu(f)=\int_X f(x) d\nu(x).$$ Moreover, given $f \in L^2(X, \nu)$, we denote its variance with respect to $\nu$ by
$${\rm Var}_\nu(f):= \int_X (f(x) - \nu(f))^2 d\nu(x) = \frac{1}{2} \int_{X \times X} (f(x) - f(y))^2 d\nu(y) d\nu(x).$$

\begin{definition} {\rm The {\it spectral gap} of $-\Delta_m$ is defined as
\begin{equation}\label{ladefdegap}{\rm gap}(-\Delta_m) := \inf \left\{ \frac{\mathcal{H}_m(f)}{{\rm Var}_\nu(f) } \ : \ f \in D(\mathcal{H}_m), \ {\rm Var}_\nu(f) \not= 0 \right\}  $$ $$= \inf \left\{ \frac{\mathcal{H}_m(f)}{\Vert f \Vert^2_2 } \ : \ f \in D(\mathcal{H}_m), \ \Vert f \Vert_2 \not= 0, \ \int_X f d\nu = 0 \right\}.
\end{equation}
}\end{definition}

Observe that, since $\nu(X) < +\infty$, we have
$$ D(\mathcal{H}_m)=L^2(X,\nu).$$

\begin{definition}\label{defpoin}{\rm
We say that $[X,d,m,\nu]$ satisfies a {\it Poincar\'{e} inequality} if there exists $\lambda >0$ such that \begin{equation}\label{Poinca1}
\lambda {\rm Var}_\nu(f) \leq \mathcal{H}_m(f) \quad \hbox{for all} \ f \in  L^2(X,\nu),
\end{equation}
or, equivalently,
$$\lambda\Vert f\Vert_{L^2(X, \nu)}^2\le \mathcal{H}_m(f)\quad  \hbox{for all} \ f \in L^2(X,\nu) \hbox{ with }\nu(f)=0.$$
}
\end{definition}

Note that, if ${\rm gap}(-\Delta_m) >0$, then $[X,d,m,\nu]$ satisfies a   Poincar\'{e} inequality with $\lambda={\rm gap}(-\Delta_m)$:
$$
{\rm gap}(-\Delta_m){\rm Var}_\nu(f) \leq \mathcal{H}_m(f) \quad \hbox{for all} \ f \in L^2(X,\nu),
$$
being the spectral gap the best constant in the Poincar\'{e} inequality.

 With such an inequality at hand and with a similar proof to the one done in the continuous setting (see, for instance, \cite{BGL}), we have  that if ${\rm gap}(-\Delta_m) >0$ then
$e^{t\Delta_{m}} u_0$ converges to $\nu(u_0)$ with exponential rate ${\rm gap}(-\Delta_m)$.

\begin{theorem}\label{PoincConv}  \vbox{The following statements are equivalent:
\begin{itemize}
\item[(i)] There exists $\lambda >0$ such that
\begin{equation}\label{DPoinca1}
\lambda {\rm Var}_\nu(f) \leq \mathcal{H}_m(f) \quad \hbox{for all} \ f \in L^2(X,\nu).
\end{equation}
\item[(ii)] For every $f \in L^2(X,\nu)$
\begin{equation}\label{1convergg}\Vert e^{t\Delta_m} f - \nu(f) \Vert_{L^2(X, \nu)} \leq  e^{- \lambda t} \Vert  f - \nu(f) \Vert_{L^2(X, \nu)}  \quad \hbox{for all} \ t \geq 0;\end{equation}
or, equivalently, for every $f \in L^2(X,\nu)$ with $\nu(f)=0$,
\begin{equation}\label{1convergg2}\Vert e^{t\Delta_m} f  \Vert_{L^2(X, \nu)} \leq  e^{-\lambda t} \Vert  f   \Vert_{L^2(X, \nu)}  \quad \hbox{for all} \ t \geq 0.\end{equation}
\end{itemize}}
\end{theorem}

\begin{remark}{\rm  Let $\Vert \mu - \nu \Vert_{TV}$ be the total variation distance:
$$\Vert \mu - \nu \Vert_{TV}:= \sup \{ \vert \mu(A) - \nu(A) \vert \ : \ A \subset X \ \hbox{Borel} \}.$$
Then, for $f \in L^2(X, \nu)$ and     $\mu_t = e^{t \Delta_m}\!f  \,\nu$, we have
\begin{equation}\label{tvorder2}\Vert \mu_t - \nu \Vert_{TV} \leq \Vert  f - 1 \Vert_{L^2(X, \nu)} \,  e^{- {\rm gap}(- \Delta_m) t}.
\end{equation}
 Indeed, by Theorem ~\ref{PoincConv}, for any Borel set $A\subset X$,
  $$\left| \int_A e^{t \Delta_m}\!fd\nu-\nu(A)\right|\le \int_A\left|e^{t \Delta_m}\!f-1\right|d\nu
\le \left(\int_X\left|e^{t \Delta_m}\!f-1\right|^2d\nu\right)^{\frac12}
\le \Vert  f - 1 \Vert_{L^2(X, \nu)} \,  e^{- {\rm gap}(- \Delta_m) t}\,.
$$
}
\end{remark}

  Hence,  it is of interest to elucidate when the spectral gap of $-\Delta_m$ is positive. In this section we will deal with such a question.

Let $H(X, \nu)$ be the subspace of $L^2(X,\nu)$ consisting of the functions which are orthonormal to the constants, i.e.,
 $$H(X, \nu)=\left\{f\in L^2(X,\nu)\,:\, \nu(f)=0\right\}.$$ Since the operator $-\Delta_m : H(X, \nu) \rightarrow H(X, \nu)$   is self-adjoint and non-negative  and $\Vert \Delta_m \Vert \leq 2$ (see Theorem \ref{generator}), by \cite[Proposition 6.9]{BrezisAF} we have that the spectrum $\sigma(-\Delta_m)$ of $-\Delta_m$  in $H(X, \nu)$ satisfies
\begin{equation}\label{breez}
 \sigma(-\Delta_m) \subset [\alpha, \beta]   \subset [0,2],
\end{equation}
 where
 $$\alpha:= \inf \left\{\langle -\Delta_m u, u \rangle \, : \, u \in H(X, \nu), \ \Vert u \Vert_2 = 1 \right\}\in \sigma(-\Delta_m),$$ and
 $$\beta:= \sup \left\{\langle -\Delta_m u, u \rangle \, : \, u \in H(X, \nu), \ \Vert u \Vert_2 = 1 \right\}  \in \sigma(-\Delta_m).$$

If $f \in L^2(X,\nu)$ and ${\rm Var}_\nu(f) \not= 0$, then $u:= f - \nu(f) \not= 0$ belongs to $H(X, \nu)$, so
$$\alpha \leq \mathcal{H}_m \left(\frac{u}{\Vert u \Vert_2} \right) = \frac{\mathcal{H}_m (u)}{\Vert u \Vert^2_2} =  \frac{\mathcal{H}_m (f)}{{\rm Var}_\nu(f)},$$
and, consequently,
\begin{equation}\label{0spectg}
{\rm gap}(-\Delta_m) = \alpha = \inf \left\{\langle -\Delta_m u, u \rangle \, : \, u \in H(X, \nu), \ \Vert u \Vert_2 = 1 \right\}.
\end{equation} Therefore,
\begin{equation}\label{spectg}
{\rm gap}(-\Delta_m) >0 \iff 0 \not\in  \sigma(-\Delta_m).
\end{equation}

If we assume that $-\Delta_m$ is the sum of an invertible and a compact operator  in $H(X, \nu)$ (this is true, for example, if the averaging operator $M_m$ is compact in $H(X,\nu)$), then, if  $0 \in   \sigma(-\Delta_m)$,  by  Fredholm's alternative Theorem, we have that there exists $u \in H(X, \nu)$, $u \not= 0$, such that $- \Delta_m u = (I - M_m) u = 0$. Then,  if $[X,d,m]$ is $m$-connected, by  Theorem \ref{ergo1}, $\Delta_m$ is ergodic so  $u$ is constant, thus $u =0$ in $H(X, \nu)$, and we get a contradiction. Consequently, we have the following result.

 \begin{proposition}\label{compact01}  Let $[X,d,m]$ be an  $m$-connected  metric random walk  space with invariant-reversible  probability measure~$\nu$.  If $-\Delta_m$ is the sum of an invertible operator and a compact operator  in $H(X, \nu)$, then ${\rm gap}(-\Delta_m) > 0$.
  \end{proposition}

\begin{example}{\rm
 (i)  If $G = (V(G), E(G))$ is a  finite weighted   connected   graph, then obviously $M_{m^G}$ is compact and, consequently, ${\rm gap}(-\Delta_m^G) > 0$. In this situation, it is well known that, for $ \sharp (V(G)) = N$, the spectrum of $- \Delta_{m^G}$ is $0 < \lambda_1 \leq \lambda_2 \leq \ldots \leq \lambda_{N-1}$  and
$0 < \lambda_1 = {\rm gap}(- \Delta_{m})$.

\item{ (ii)}
  Another example in which $-\Delta_m$ is the sum of an invertible and a compact operator is $[\Omega, d, m^{J,\Omega}]$ with $\Omega$ a bounded domain and the kernel $J$ satisfying: $J \in C(\R^N, \R)$ is nonnegative, radially symmetric with $J(0) >0$ and $\int_{\R^N} J(x) dx =1$.
Indeed,
$$-\Delta_{m^{J,\Omega}}f(x)=\int_\Omega J(x-y)dyf(x)-\int_\Omega J(x-y)f(y)dy,$$
where $\displaystyle\int_\Omega J(x-y)dyf(x)$ defines an invertible operator and $\displaystyle\int_\Omega J(x-y)f(y)dy$ defines a compact operator.
Hence, in this case we have (see also~\cite{ElLibro}):
$${\rm gap}(- \Delta_{m^{J,\Omega}}) =\inf \left\{ \frac{\displaystyle\frac{1}{2} \displaystyle\int_{\Omega\times\Omega} J(x-y) (u(y) - u(x))^2 dxdy}{\displaystyle\int_\Omega u(x)^2 dx} \ : \ u \in L^2(\Omega), \Vert u\Vert_{L^2(X,\nu)}>0,  \int_\Omega u = 0 \right\} > 0.$$
  Let us point out that the condition $J(0)>0$ is necessary, see~\cite[Remark 6.20]{ElLibro}.}

 \end{example}

 As a consequence of a result by Miclo \cite{Miclo}, we have that ${\rm gap}(- \Delta_{m})  >0$ if $\Delta_{m}$ is ergodic and $M_m$ is hyperbounded, that is, if there exists $p >2$ such that $M_m$ is bounded from $L^2(X,\nu)$ to $L^p(X,\nu)$. If we have that $m_x \ll \nu$, i.e.,  $m_x = f_x \nu$  with $f_x \in L^1(X, \nu)$, and we assume that
\begin{equation}\label{hyperbounded}
\int_X \Vert f_x \Vert_{L^2(X,\nu)}^p d\nu(x) = K < \infty,
\end{equation}
then, for $u \in L^2(X,\nu)$, by the Cauchy-Schwarz inequality, we have that
$$\Vert M_m u \Vert_p^p = \int_X \vert M_m u(x) \vert^p d\nu(x) = \int_X \left\vert \int_X u(y) dm_x(y) \right\vert^p d\nu(x) = \int_X \left\vert \int_X u(y) f_x(y) d\nu(y) \right\vert^p d\nu(x)$$ $$ \leq \Vert u \Vert_{L^2(X,\nu)}^p  \int_X  \Vert f_x \Vert_{L^2(X,\nu)}^p d\nu(x),$$
hence
$$\Vert M_m u \Vert_p \leq K^{\frac{1}{p}}\Vert u \Vert_{L^2(X,\nu)}.$$
Therefore, $M_m$ is hyperbounded and, consequently, we have the following result about the spectral gap.
\begin{proposition} If $\Delta_{m}$ is ergodic and~\eqref{hyperbounded} holds, then ${\rm gap}(- \Delta_{m})  >0.$
\end{proposition}

In the next example we   see that there exist  metric random walk spaces for which the Poincar\'{e} inequality does not hold.

\begin{example}\label{lexam34}{\rm
Let $V(G)=\{x_3,x_4,x_5\ldots,x_n\ldots \}$ be a weighted linear graph with
$$w_{x_{3n},x_{3n+1}}=\frac{1}{n^3} , \  w_{x_{3n+1},x_{3n+2}}=\frac{1}{n^2}, \ w_{x_{3n+2},x_{3n+3}}=\frac{1}{n^3} , $$
for  $n\geq 1$, and let
$$f_n(x)=\left\{ \begin{array}{ll} \displaystyle n \quad &\hbox{if} \ \ x=x_{3n+1},x_{3n+2} \\ \\ 0 \quad &\hbox{else}. \end{array}\right.
$$
Note that $\nu(X)<+\infty$ (we avoid its normalization for simplicity). Now,
$$2\mathcal{H}_m(f_n)=\int_X\int_X (f_n(x)-f_n(y))^2 dm_x(y)d\nu(x)  $$
$$ =d_{x_{3n}}\int_X (f_n(x_{3n})-f_n(y))^2dm_{x_{3n}}(y)+d_{x_{3n+1}}\int_X (f_n(x_{3n+1})-f_n(y))^2dm_{x_{3n+1}}(y)$$
$$+d_{x_{3n+2}}\int_X (f_n(x_{3n+2})-f_n(y))^2dm_{x_{3n+2}}(y)+d_{x_{3n+3}}\int_X (f_n(x_{3n+3})-f_n(y))^2dm_{x_{3n+3}}(y)$$
$$=d_{x_{3n}}n^2\frac{\frac{1}{n^3}}{d_{x_{3n}}}+d_{x_{3n+1}}n^2\frac{\frac{1}{n^3}}{d_{x_{3n+1}}}+d_{x_{3n+2}}n^2\frac{\frac{1}{n^3}}{d_{x_{3n+2}}}+d_{x_{3n+3}}n^2\frac{\frac{1}{n^3}}{d_{x_{3n+3}}}=\frac{4}{n} .$$
However, we have
$$\int_X f_n(x)d\nu(x)=n (d_{x_{3n+1}}+d_{x_{3n+2}})=2n \left(\frac{1}{n^2}+\frac{1}{n^3}\right)=\frac{2}{n}\left(1+\frac{1}{n}\right),$$
thus
$$\nu(f_n)=\frac{\frac{2}{n}\left(1+\frac{1}{n}\right)}{\nu(X)}={\widetilde O}   \left(\frac{1}{n}\right),$$
where we use the notation $$\varphi(n) = {\widetilde O}   (\psi(n)) \iff \exists \lim_{n \to \infty} \frac{\varphi(n)}{\psi(n)} = C \neq 0.$$
Therefore,
$$(f_n(x)-\nu(f_n))^2=\left\{ \begin{array}{ll} \displaystyle {\widetilde O}   (n^2) \quad &\hbox{if} \ \ x=x_{3n+1},x_{3n+2}, \\ \\ {\widetilde O}   \left(\frac{1}{n^2}\right) \quad &\hbox{otherwise}. \end{array}\right.
$$
Finally,
$${\rm Var}_\nu(f_n)=\int_X(f_n(x)-\nu(f_n))^2d\nu(x)={\widetilde O}   \left(\frac{1}{n^2}\right)\sum_{x\neq x_{3n+1},x_{3n+2}}d_{x}+ {\widetilde O}   (n^2)(d_{x_{3n+1}}+d_{x_{3n+2}})$$
$$={\widetilde O}   \left(\frac{1}{n^2}\right)+2 {\widetilde O}   (n^2)\left(\frac{1}{n^2}+\frac{1}{n^3}\right)={\widetilde O}   (1) .$$
Consequently, $[V(G),d_G,(m_x),\nu]$ does not satisfy a  Poincar\'{e} inequality for any $\lambda>0$.}
\end{example}

 In general,  since $\mathcal{H}_m(f)=-\int_Xf(x)\Delta_mf(x)d\nu(x)$, if  $\Delta_m f = 0$ then $\mathcal{H}_m(f)=0$ and, therefore, if $[X,d,m,\nu]$ satisfies a   Poincar\'{e} inequality,  we have that $f$ is constant:
$$f(x)= \int_X f(x) d\nu(x)\quad\nu-\hbox{a.e.}$$
 Consequently, we get the following result.
\begin{proposition}
If $[X,d,m,\nu]$ satisfies a   Poincar\'{e} inequality we have that $\Delta_m$ is ergodic.
 \end{proposition}
  Example~\ref{lexam34} shows that the reverse implication does not hold in general.

\subsection{ Isoperimetric Inequality}

 Recall that, for a $\nu$-measurable set $D\subset X$,
$$P_m(D)=\int_E\int_{X\setminus E}dm_x(y)d\nu(x)=TV_m(\1_E) . $$
The  Poincar\'{e} inequality, if given only for characteristic functions, implies that there exists $\lambda>0$ such that
 \begin{equation}\label{relwiso}\lambda\,\nu(D)\big(1-\nu(D)\big)\le P_m(D)\quad\hbox{for every $\nu-$measurable set $D$},
 \end{equation}
(observe that  this also implies  the ergodicity of $\Delta_m$, as we have seen in Theorem \ref{ERggB}).
Hence,  since $$\hbox{min}\{x,1-x\}\le 2x(1-x)\le 2\hbox{min}\{x,1-x\} \quad  \hbox{for} \ 0\le x\le 1,$$ inequality~\eqref{relwiso} implies the following isoperimetric inequality (see \cite[Theorem 3.46]{AFP}):
\begin{equation}\label{isop01}\hbox{min}\big\{\nu(D),1-\nu(D)\big\}
\le\frac{2}{\lambda}P_m(D)\quad\hbox{for every $\nu-$measurable set $D$};
\end{equation}
and, conversely, the isoperimetric inequality \eqref{isop01} implies
$$\frac{\lambda}{2}\,\nu(D)\big(1-\nu(D)\big)\le P_m(D)\quad\hbox{for every $\nu-$measurable set $D$}.$$

 \begin{definition}
   If there exists $\lambda >0$ satisfying \eqref{isop01}, we say that $[X,d,m, \nu]$ satisfies an {\it isoperimetric inequality}.
\end{definition}

\subsection{Cheeger Inequality}
In a weighted graph $G=(V(G), E(G))$ the {\it Cheeger constant} is defined  as
 \begin{equation}\label{chconst}
 h_G:= \inf_{ D \subset V(G)} \ \frac{\vert \partial D \vert}{\min\{ \nu_G(D), \nu_G(V(G) \setminus D)\}},
 \end{equation}
 where $$\vert \partial D \vert:= \sum_{x \in D, y \in V \setminus D} w_{xy} .$$
 In \cite{Ch} (see also \cite{BJ}), the following relation between the Cheeger constant and the first positive eigenvalue $\lambda_1(G)$ of the graph Laplacian $\Delta_{m^G}$ is proved:
 \begin{equation}\label{chconstN}
 \frac{h_G^2}{2}\leq \lambda_1(G) \leq 2  h_G.
  \end{equation}
  The previous inequality appeared in \cite{Cheeger}, and can be traced
back to the paper by Polya and Szego \cite{PS}.

Let $[X, d, m]$ be a metric random walk space with invariant and reversible  probability measure~$\nu$.  We define its {\it Cheeger constant} as
$$
 h_m(X):= \inf \left\{ \frac{P_m (D)}{\min\{ \nu(D), \nu(X \setminus D)\}} \ : \  D \subset X, \ 0 < \nu(D) < 1 \right\}, $$
 or, equivalently, $$h_m(X)= \inf \left\{ \frac{P_m (D)}{\nu(D)} \ : \  D \subset X, \ 0 < \nu(D) \leq \frac{1}{2}\right\}.
$$
Having in mind \eqref{perim}, we have that this definition is consistent with the definition on graphs. Note that, if $h_m(X) >0$, then $h_m(X)$ is the best constant in the isoperimetric inequality \eqref{isop01}.

   We will now  give a variational characterization of the Cheeger constant  which generalizes the one obtained in \cite{SB1} for the particular case of finite graphs. Recall that, given a function $u : X \rightarrow \R$,  $\mu \in \R$ is a {\it median} of $u$ with respect to a measure $\nu$ if
$$\nu(\{ x \in X \ : \ u(x) < \mu \}) \leq \frac{1}{2} \nu(X), \quad \nu(\{ x \in X \ : \ u(x) > \mu \}) \leq \frac{1}{2} \nu(X).$$
We denote by ${\rm med}_\nu (u)$ the set of all medians of $u$. It is easy to see that
$$\mu \in {\rm med}_\nu (u) \iff - \nu(\{ u = \mu \}) \leq \nu(\{ x \in X \ : \ u(x) > \mu \}) - \nu(\{ x \in X \ : \ u(x) < \mu \}) \leq \nu(\{ u = \mu \}),$$
from where it follows that
\begin{equation}\label{sigg1}
0 \in {\rm med}_\nu (u) \iff \exists \xi \in {\rm sign}(u) \ \hbox{such that} \ \int_X \xi(x) d \nu(x) = 0,
\end{equation}
where
$${\rm sign}(u)(x):=  \left\{ \begin{array}{lll} 1 \quad \quad &\hbox{if} \ \  u(x) > 0, \\ -1 \quad \quad &\hbox{if} \ \ u(x) < 0, \\ \left[-1,1\right] \quad \quad &\hbox{if} \ \ \ u = 0. \end{array}\right.$$

Let
\begin{equation}\label{minnb}\lambda_1^m(X) := \inf \left\{ TV_m(u) \ : \ \Vert u \Vert_1 = 1, \ 0 \in {\rm med}_\nu (u) \right\}.\end{equation}

\begin{theorem}\label{varchar} If $[X, d, m]$ is a metric random walk space with invariant and reversible  probability measure~$\nu$, then
$$h_m(X) =\lambda_1^m(X).$$
\end{theorem}
\begin{proof} If $D \subset X, \ 0 < \nu(D) \leq \frac{1}{2}$, then $0 \in {\rm med}_\nu (\1_D)$. Thus,
$$\lambda_1^m(X) \leq TV_m \left(\frac{1}{\nu(D)} \1_D \right) = \frac{1}{\nu(D)} P_m (D)$$
and, therefore,
$$\lambda_1^m(X) \leq  h_m(X).$$
Now, for the other inequality, let $u \in L^1(X, \nu)$ such that $\Vert u \Vert_1 = 1$ and $0 \in {\rm med}_\nu (u)$. Since $0 \in {\rm med}_\nu (u)$, by the  Coarea formula (Theorem \ref{coarea11}), and having in mind that the set  $\{ t \in \R \ : \ \nu( \{ u = t \}) > 0 \}$ is countable, we have
$$\begin{array}{c}
\displaystyle TV_m(u) = \int_{-\infty}^{+\infty} P_m(E_t(u))\, dt = \int_{0}^{+\infty} P_m(E_t(u))\, dt + \int_{-\infty}^{0} P_m(X \setminus E_t(u))\, dt
\\ \\ \displaystyle \geq h_m(X) \int_{0}^{+\infty} \nu(E_t(u))\, dt + h_m(X) \int_{-\infty}^{0} \nu(X \setminus E_t(u))\, dt\\ \\ \displaystyle
 = h_m(X) \left(\int_X u^+(x) d\nu(x) + \int_X u^-(x) d\nu(x) \right) = h_m(X) \Vert u \Vert_1 = h_m(X).
 \end{array}
 $$
Therefore, taking the infimum in $u$, we get $\lambda_1^m(X) \geq  h_m(X).$
\end{proof}

  Following~\cite{Ch} and using Theorem \ref{varchar}, in the next result we see that the Cheeger inequality \eqref{chconstN} also holds in our context.

\begin{theorem}\label{isopoin} Let $[X, d, m]$ be a metric random walk space with invariant and reversible  probability measure~$\nu$. The following Cheeger inequality holds
\begin{equation}\label{CHEEGER1}\frac{h^2_m}{2} \leq {\rm gap}(-\Delta_m) \leq 2 h_m.
\end{equation}
\end{theorem}

 \begin{proof} Let $(f_n)\subset D(\mathcal{H}_m)$, with $\nu( f_n) = 0$, such that
 $$\lim_{n \to \infty} \frac{\mathcal{H}_m(f_n)}{\Vert f_n \Vert^2_2}  = {\rm gap}(-\Delta_m).$$
If we take $\mu_n \in {\rm med}_\nu(f_n)$, we have
\begin{equation}\begin{array}{l}\displaystyle 2 \mathcal{H}_m(f_n) = \int_X \int_X (f_n(y) - \mu_n -(f_n(x) - \mu_n))^2dm_x(y) d\nu(x)
\\ \\ \displaystyle =\int_X \int_X  \left[(f_n(y) - \mu_n)^+ -(f_n(x) - \mu_n)^+ - ((f_n(y) - \mu_n)^- -(f_n(x) - \mu_n)^-) \right]^2dm_x(y) d\nu(x) \\ \\ \displaystyle
=\int_X \int_X  \left((f_n(y) - \mu_n)^+ -(f_n(x) - \mu_n)^+ \right)^2 dm_x(y) d\nu(x)  \\ \\ \displaystyle
\quad + \int_X \int_X  \left( (f_n(y) - \mu_n)^- -(f_n(x) - \mu_n)^- \right)^2dm_x(y) d\nu(x)\\ \\ \displaystyle
 \quad - 2\int_X \int_X \left((f_n(y) - \mu_n)^+ -(f_n(x) - \mu_n)^+ \right)  \left( (f_n(y) - \mu_n)^- -(f_n(x) - \mu_n)^- \right) dm_x(y)  d\nu(x) \ .
\end{array}\end{equation}
Now, an easy calculation gives
$$ - \int_X \int_X \left((f_n(y) - \mu_n)^+ -(f_n(x) - \mu_n)^+ \right)  \left( (f_n(y) - \mu_n)^- -(f_n(x) - \mu_n)^- \right) dm_x(y) d\nu(x) \geq 0.$$
On the other hand, since $\nu( f_n) = 0$, we have
$$\int_X f_n^2(x) d \nu(x) \leq \int_X (f_n(x) - \mu_n)^2 d \nu(x).$$
Therefore,
\begin{equation}\begin{array}{l}\displaystyle\frac{2\mathcal{H}_m(f_n)}{\Vert f_n \Vert^2_2} \geq \frac{\displaystyle\int_X \int_X  \left((f_n(y) - \mu_n)^+ -(f_n(x) - \mu_n)^+ \right)^2 dm_x(y) d\nu(x)}{\displaystyle\int_X[(f_n(x) - \mu_n)^+]^2 d\nu(x) + \int_X[(f_n(x) - \mu_n)^-]^2 d\nu(x)}+  \\ \\ \displaystyle
\qquad \qquad \qquad + \frac{\displaystyle\int_X \int_X  \left( (f_n(y) - \mu_n)^- -(f_n(x) - \mu_n)^- \right)^2dm_x(y) d\nu(x)}{\displaystyle\int_X[(f_n(x) - \mu_n)^+]^2 d\nu(x) + \int_X[(f_n(x) - \mu_n)^-]^2 d\nu(x)} \ .
\end{array}\end{equation}
Having in mind that
$$\frac{a+b}{c+d} \geq \min \left\{ \frac{a}{c}, \frac{b}{d} \right\} \quad \hbox{for every} \ a,b,c,d\in\R^+, $$
and
$$\int_X[(f_n(x) - \mu_n)^+]^2 d\nu(x) + \int_X[(f_n(x) - \mu_n)^-]^2 d\nu(x)>0,$$
we can assume, without loss of generality, that $$\int_X[(f_n(x) - \mu_n)^+]^2 d\nu(x) > 0,$$
and that
$$\frac{2\mathcal{H}_m(f_n)}{\Vert f_n \Vert^2_2} \geq  \frac{\displaystyle\int_X \int_X  \left((f_n(y) - \mu_n)^+ -(f_n(x) - \mu_n)^+ \right)^2 dm_x(y) d\nu(x)}{\displaystyle\int_X[(f_n(x) - \mu_n)^+]^2 d\nu(x) }.$$
By the Cauchy-Schwartz inequality, we have
\begin{equation}\begin{array}{l}\displaystyle\int_X \int_X  \left\vert [(f_n(y) - \mu_n)^+]^2 -[(f_n(x) - \mu_n)^+]^2 \right\vert dm_x(y)d\nu(x)  \\ \\ \displaystyle
= \int_X \int_X  \left\vert (f_n(y) - \mu_n)^+ -(f_n(x) - \mu_n)^+ \right\vert \left\vert(f_n(y) - \mu_n)^+ +(f_n(x) - \mu_n)^+ \right\vert dm_x(y)d\nu(x) \\ \\ \displaystyle
\leq \left(\int_X \int_X  \left((f_n(y) - \mu_n)^+ -(f_n(x) - \mu_n)^+ \right)^2 dm_x(y) d\nu(x)\right)^{\frac12} \times \\ \\ \displaystyle
\qquad  \times \left(\int_X \int_X  \left((f_n(y) - \mu_n)^+ +(f_n(x) - \mu_n)^+ \right)^2 dm_x(y) d\nu(x)\right)^{\frac12} \,.
\end{array}\end{equation}
Now, by the invariance of $\nu$,
$$\int_X \int_X  \left((f_n(y) - \mu_n)^+ +(f_n(x) - \mu_n)^+ \right)^2 dm_x(y) d\nu(x)\leq 4 \int_X[(f_n(x) - \mu_n)^+]^2 d\nu(x).$$
Thus
$$\frac{2\mathcal{H}_m(f_n)}{\Vert f_n \Vert^2_2} \geq \left( \frac{\frac12 \displaystyle\int_X \int_X  \left\vert [(f_n(y) - \mu_n)^+]^2 -[(f_n(x) - \mu_n)^+]^2 \right\vert dm_x(y)d\nu(x)}{\displaystyle\int_X[(f_n(x) - \mu_n)^+]^2 d\nu(x)} \right)^{2}. $$
Then, since $0\in {\rm med}_\nu([(f_n - \mu_n)^+]^2)$, by Theorem \ref{varchar}, we get
$$\frac{2\mathcal{H}_m(f_n)}{\Vert f_n \Vert^2_2} \geq h_m(X)^2,$$
and, consequently, taking limits as $n \to \infty$, we obtain
$$\frac{h^2_m}{2} \leq {\rm gap}(-\Delta_m).$$

To prove the other inequality  we can assume that ${\rm gap}(-\Delta_m) >0$.  Now, by \eqref{isop01}, we have
\begin{equation}\label{WWisop01}\hbox{min}\big\{\nu(D),1-\nu(D)\big\}
\le\frac{2}{{\rm gap}(-\Delta_m)}P_m(D) \quad \hbox{for all} \ D \subset X, \ 0 < \nu(D) < 1,
\end{equation}
from where it follows that ${\rm gap}(-\Delta_m) \leq 2 h_m(X)$.
 \end{proof}

Let $A \subset X$ with $\nu(A) = \frac12$ and $u= \1_A - \1_{X \setminus A}$. It is easy to see that
$TV_m(u)=2P_m(A)$ and $\mathcal{H}_m(u)=4P_m(A)$. Hence, since $\Vert u\Vert_1=\Vert u\Vert_2=1$, $\nu(u)=0$ and $0 \in {\rm med}_\nu (u)$, we obtain the following result as a consequence of~Theorem \ref{varchar}.

 \begin{corollary}\label{abovecor} Let $[X, d, m]$ be a metric random walk space with invariant and reversible  probability measure~$\nu$. Let $A \subset X$ with $\nu(A) = \frac12$ and $u= \1_A - \1_{X \setminus A}$. Then,

 1. $h_m(X) = \frac{P_m(A)}{\nu(A)} \iff u= \1_A - \1_{X \setminus A} \ \hbox{ is a minimizer of } \ \eqref{minnb}$.

2. $u$ is a minimizer of~\eqref{minnb} and ${\rm gap}(-\Delta_m) = 2 h_m(X)$ $\iff$
   $u$ is a minimizer of~\eqref{ladefdegap}.
\end{corollary}

  Bringing together all the above results we have:

\begin{theorem}\label{cheegerDD} Let $[X, d, m]$ be a metric random walk space with invariant and reversible  probability measure~$\nu$. The following statements are equivalent:
\begin{enumerate}
\item $[X,d,m, \nu]$ satisfies a Poincar\'{e} inequality,
\\
\item  ${\rm gap}(-\Delta_m) > 0$,
\\
\item  $[X,d,m, \nu]$ satisfies an  isoperimetric inequality,
\\
\item  $h_m(X) > 0$.
\end{enumerate}
\end{theorem}

\begin{example}{\rm It is well known, (see for instance \cite{Ch}) that for finite graphs $G$, $h_m(G) > 0$ if, and only if, $G$ is connected. This result is not true for infinite graphs. In fact, the graph of the Example \ref{lexam34} is connected and its Cheeger constant is zero since its spectral gap is zero.
}
\end{example}

\subsection{ Spectral Gap and Curvature}
   Since $\mathcal{E}_m$ admits a {\it Carr\'{e} du champ} $\Gamma$ (see \cite{BGL}) defined by
$$\Gamma(f,g)(x) := \frac{1}{2} \Big(\Delta_m(fg)(x) - f(x)\Delta_mg(x) - g(x) \Delta_mf (x) \Big) \ \ \hbox{for all } x\in X \hbox{ and } f,g \in L^2(X, \nu),$$
we  can study the Bakry-\'{E}mery curvature-dimension condition in this context. We will study its relation with the spectral gap.

According to Bakry and \'{E}mery \cite{BE}, we define the Ricci curvature operator~$\Gamma_2$ by iterating~$\Gamma$:
$$\Gamma_2(f,g)   := \frac{1}{2} \Big( \Delta_m \Gamma (f,g)  - \Gamma (f,\Delta_m g)- \Gamma ( \Delta_m f,g)  \Big),$$
 which is well defined for $f,g\in L^2(X, \nu)$.
We will write, for $f\in L^2(X, \nu)$,
$$\Gamma(f):=\Gamma(f,f)=\frac12 \Delta_m(f^2)-f\Delta_mf $$
and
$$\Gamma_2(f):= \Gamma_2(f,f)= \frac{1}{2}  \Delta_m \Gamma(f) -   \Gamma(f,\Delta_m f).$$
It is easy to see that
$$\Gamma(f,g)(x)=\frac12\int_X\nabla f(x,y)\nabla g(x,y)dm_x(y) \quad \hbox{ and } \quad \Gamma(f)(x) = \frac{1}{2} \int_X \vert \nabla f(x,y) \vert^2 dm_x(y).$$
Consequently,
\begin{equation}\label{good3}
\int_X \Gamma(f,g)(x) d\nu(x)  =  \mathcal{E}_m (f,g) \quad \hbox{ and } \quad \int_X \Gamma(f)(x) d\nu(x)  =  \mathcal{H}_m (f).
\end{equation}
Furthermore,   by  \eqref{Lap0} and \eqref{good3}, we get
$$\int_X \Gamma_2(f)\, d\nu = \frac{1}{2} \int_X \left(\Delta_m \Gamma(f) - 2 \Gamma(f,\Delta_m f) \right) \, d \nu  = - \int_X \Gamma(f,\Delta_m f) \, d\nu   = -  \mathcal{E}_m (f,\Delta_m f),$$ thus
\begin{equation}\label{intg22116} \int_X \Gamma_2(f)\, d\nu=  \int_X (\Delta_m f)^2 \, d\nu.
\end{equation}

\begin{definition}{\rm
The operator $\Delta_m$ satisfies the {\it Bakry-\'{E}mery curvature-dimension condition}   $BE(K, n)$ for $n \in (1, +\infty)$ and $K \in \R$ if
$$\Gamma_2(f) \geq \frac{1}{n} (\Delta_m f)^2 + K \Gamma(f) \quad \forall \, f \in L^2(X, \nu).$$
The constant  $n$ is the {\it dimension of the operator} $\Delta_m$, and $K$ is the {\it lower bound of the Ricci curvature of the operator}~$\Delta_m$.
If there exists  $K \in \R$ such that
$$\Gamma_2(f) \geq  K \Gamma(f) \quad \forall \, f \in L^2(X, \nu),$$
then  it is said that  the operator $\Delta_m$ satisfies the {\it Bakry-\'{E}mery curvature-dimension condition}   $BE(K, \infty)$.
}
\end{definition}

Observe that if $\Delta_m$ satisfies the {\it Bakry-\'{E}mery curvature-dimension condition}   $BE(K, n)$  then     it also satisfies the {\it Bakry-\'{E}mery curvature-dimension condition}   $BE(K, m)$
for $m> n$.

This definition is motivated by the well known fact that on a complete  $n$-dimen\-sional Riemannian  manifold $(M, g)$, the Laplace-Beltrami operator $\Delta_g$ satisfies $BE(K, n)$ if, and only if, the Ricci curvature of the Riemannian manifold is bounded from below by~$K$ (see, for example, \cite[Appendix C.6]{BGL}).

 The use of the Bakry-\'{E}mery curvature-dimension condition as a possible definition of a Ricci curvature bound in Markov chains was first considered  in 1998 \cite{S}. Now,  this concept of  Ricci curvature in the discrete setting has been frequently used since the work by Lin and  Yau   \cite{LY} (see \cite{KKRT} and the references therein).

  Integrating the   Bakry-\'{E}mery curvature-dimension condition    $BE(K, n)$ we have
$$\int_X \Gamma_2(f)\, d\nu  \geq \frac{1}{n} \int_X (\Delta_m f)^2 \, d\nu + K  \int_X \Gamma(f) \, d\nu.  $$
Now, by  \eqref{good3} and \eqref{intg22116}, this inequality can be rewritten as
$$ \int_X (\Delta_m f)^2 \, d\nu \geq \frac{1}{n} \int_X (\Delta_m f)^2 \, d\nu + K \mathcal{H}_m (f),$$
or, equivalently, as
\begin{equation}\label{bkinteg} K\frac{n}{n-1} \mathcal{H}_m (f) \leq \int_X (\Delta_m f)^2 \, d\nu.
\end{equation}
Similarly, integrating the   Bakry-\'{E}mery curvature-dimension condition    $BE(K,  \infty)$ we have
\begin{equation}\label{bkinteginf} K \mathcal{H}_m (f) \leq  \int_X (\Delta_m f)^2 \, d\nu.
\end{equation}

 We call the inequalities \eqref{bkinteg} and \eqref{bkinteginf} the {\it integrated Bakry-\'{E}mery curvature-dimension conditions}, and denote them by $IBE(K,n)$ and $IBE(K,\infty)$, respectively.

 \begin{theorem}\label{cotagap} Let $[X, d, m]$ be a  metric random walk space with invariant-reversible  probability measure~$\nu$. Assume that $\Delta_m$ is ergodic. Then
\begin{equation}\label{spectgap1}
{\rm gap}(-\Delta_m) = \sup \Big\{ \lambda \geq 0 \, : \, \lambda \mathcal{H}_m(f) \leq \int_X (-\Delta_m f)^2 d\nu \ \ \forall f \in L^2(X,\nu) \Big\}.
\end{equation}
\end{theorem}

\begin{proof} By \eqref{0spectg} we know that ${\rm gap}(-\Delta_m) =\alpha$. Set
$$A:= \sup \Big\{ \lambda \geq 0 \ : \ \lambda \mathcal{H}_m(f) \leq \int_X (-\Delta_m f)^2 d\nu \ \ \forall f \in L^2(X,\nu) \Big\}.$$
 Let $(P_\lambda)_{\lambda \geq 0}$ be the spectral projection of the  self-adjoint and positive operator  $-\Delta_m : H(X, \nu) \rightarrow H(X, \nu)$. By the spectral Theorem \cite[Theorem VIII. 6]{RS}, we have
\begin{equation}\label{spectrrall}
\begin{array}{c}\displaystyle\mathcal{H}_m(f) = \langle -\Delta_m f, f \rangle = \int_\alpha^\beta \lambda d \langle P_\lambda f, f \rangle \\ \\
\displaystyle \int_X (-\Delta_m f)^2 d\nu = \langle -\Delta_m f,  -\Delta_m f \rangle = \int_\alpha^\beta \lambda^2 d \langle P_\lambda f, f \rangle.
\end{array}
\end{equation}
Hence, since $\lambda^2 \geq \alpha \lambda$, we have that
$$ \int_X (-\Delta_m f)^2 d\nu \geq \alpha \int_\alpha^\beta \lambda d \langle P_\lambda f, f \rangle = \alpha \mathcal{H}_m(f),$$
and we get $\alpha \leq A$. Finally, let us see that $\alpha \geq A$.  Since $\alpha\in \sigma(\Delta_m)$, given $\epsilon > 0$, there exists $0 \not= f \in {\rm Range}(P_{\alpha + \epsilon})$ and, consequently, $P_\lambda f = f$ for $\lambda \geq\alpha + \epsilon$. Then, having in mind the ergodicity of $\Delta_m$, we have
\begin{equation}\label{mdp01prev} 0 <\int_X (-\Delta_m f)^2 d\nu  = \int_\alpha^{\alpha + \epsilon} \lambda^2 d \langle P_\lambda f, f \rangle \leq (\alpha + \epsilon) \int_\alpha^{\alpha + \epsilon} \lambda d \langle P_\lambda f, f \rangle = (\alpha + \epsilon)\mathcal{H}_m(f)<(\alpha + 2\epsilon)\mathcal{H}_m(f) .\end{equation}
This implies that $\alpha+2\epsilon$ does not belong to  the set
$$\Big\{ \lambda \geq 0 \ : \ \lambda \mathcal{H}_m(f) \leq  \int_X (-\Delta_m f)^2 d\nu \ \ \forall f \in L^2(X,\nu) \Big\},$$  thus
$A<\alpha+2\epsilon.$
Therefore, since $\epsilon >0$ was arbitrary, we have
\[ A\le  \alpha. \qedhere\]
\end{proof}

 Consequently, on account of Theorem~\ref{cotagap}, we  can rewrite the Poincar\'{e} inequality via the integrated Bakry-\'{E}mery curvature-dimension condition (see \cite[Theorem 4.8.4]{BGL}, see also~\cite[Theorem 2.1]{BCLL}):

\begin{theorem}\label{madrid1} Let $[X, d, m]$ be a  metric random walk space with invariant-reversible  probability  measure $\nu $. Assume that $\Delta_m$ is ergodic. Then,
\begin{itemize}
 \item[(1)] $\Delta_m$ satisfies an integrated Bakry-\'{E}mery curvature-dimension condition  $IBE(K, n)$ with $K >0$ if, and only if, a Poincar\'{e} inequality with constant $K\frac{n}{n-1}$ is satisfied.

      \item[(2)] $\Delta_m$ satisfies an integrated Bakry-\'{E}mery curvature-dimension condition  $IBE(K, \infty)$ with $K >0$  if, and only if, a Poincar\'{e} inequality with constant $K$ is satisfied.
\end{itemize}
Therefore,
 if  $\Delta_m$ satisfies  the  Bakry-\'{E}mery curvature-dimension condition $BE(K,n)$ with $K>0$, we have
\begin{equation}\label{ole1}{\rm gap}(- \Delta_m) \geq K\frac{n}{n-1}.\end{equation}
In the case that $\Delta_m$ satisfies  the     Bakry-\'{E}mery curvature-dimension condition $BE(K,\infty)$ with $K >0$, we have
\begin{equation}\label{ole2}{\rm gap}(- \Delta_m) \geq K.
\end{equation}
\end{theorem}

 In the next example we will see that, in general, the integrated  Bakry-\'{E}mery curvature-dimension condition  $IBE(K,n)$ with $K >0$ does not imply the Bakry-\'{E}mery curvature-dimension condition  $BE(K,n)$ with $K >0$.

\begin{example}\label{megusta1921}{\rm Consider the weighted linear graph $G$ with vertex set $V(G) = \{a,b,c\}$ and where the only non-zero weights are $w_{a,b} =w_{b,c}=1$,  and let $\Delta := \Delta_{m^G}$.  A simple calculation gives
$$\Gamma(f) (a) = \frac12(f(b) - f(a))^2 = \frac12 (\Delta f (a))^2,$$ $$ \Gamma(f) (c) = \frac12(f(b) - f(c))^2 = \frac12 (\Delta f (c))^2,$$
\begin{equation}\label{ccc06}\Gamma(f) (b) = \frac14(f(b) - f(a))^2 + \frac14 (f(b) -f(c))^2=\frac14\left((\Delta f (a))^2+(\Delta f (c))^2\right) = \frac12\left(\Gamma(f) (a)+\Gamma(f) (c)\right).\end{equation}
 Moreover,
 \begin{equation}\label{aa1}\Gamma_2(f) (a) = \frac18 (\Delta f (c))^2 +\frac58 (\Delta f (a))^2+\frac14 \Delta f(a)\Delta f(c)
 \end{equation}
and
  \begin{equation}\label{cc1}\Gamma_2(f) (c) = \frac18 (\Delta f (a))^2 +\frac58 (\Delta f (c))^2+\frac14 \Delta f(a)\Delta f(c).
 \end{equation}
 Having in mind \eqref{aa1} and  \eqref{cc1}, the $BE(K,n)$ condition
 $$\Gamma_2(f) \geq \frac{1}{n} (\Delta  f)^2 + K \Gamma(f) \quad \forall \, f \in L^2(X, \nu)$$
on $a$ or $c$ holds true if, and only if,
\begin{equation}\label{cccc01}\frac14 y^2+\frac54x^2+ \frac12x y
 \geq \frac{2}{n} x^2 + K  x^2\qquad\forall x,y\in  \mathbb{R}.\end{equation}
 Now,  since~\eqref{cccc01}   is true for $x=0$,  \eqref{cccc01} holds if, and only if,
 $$K \leq \inf_{x\neq 0,\, y} \frac{\frac14 y^2+\frac54x^2+\frac12x y -\frac{2}{n} x^2}{x^2}.$$
 Moreover, taking $y=\lambda x$, we obtain that the following inequality must be satisfied
 $$K \leq \inf_{ \lambda }\Big(\frac14\lambda^2+\frac54 + \frac12\lambda - \frac{2}{n}\Big) = 1- \frac2n.$$
In fact,  it is easy to see that~\eqref{cccc01} is true for any $K\le 1- \frac2n$.

On the other hand, we have that
 \begin{equation}\label{bb1}\Gamma_2(f) (b) =  \frac12 (\Delta f(b))^2 + \Gamma(f)(b),
 \end{equation}
and it is easy to see that
 \begin{equation}\label{bb1bisel}\Gamma_2(f) (b)   \geq \frac1n (\Delta f(b))^2 + K \Gamma(f)(b) \ \ \ \hbox{for all} \ n >1, \ K \leq 1-\frac{2}{n}.
 \end{equation}
Therefore, we have that this graph Laplacian satisfies  the   Bakry-\'{E}mery curvature-dimension condition $$BE\left(1- \frac2n,n\right) \quad  \hbox{for any} \ \  n >1 ,$$
 being $K=1-\frac2n$ the best constant for a fixed $n>1$.

 Now, it is easy to see that ${\rm gap}(- \Delta) = 1$ thus, by Theorem \ref{madrid1}, we have that $\Delta$ satisfies  the  {\it integrated Bakry-\'{E}mery curvature-dimension condition}   $IBE(K, n)$ with $K= 1 -\frac{1}{n} > 1- \frac2n$.

 Note that $\Delta$ satisfies  the   Bakry-\'{E}mery curvature-dimension condition $BE(1,\infty)$ and hence, in this example, the bound in \eqref{ole2} is sharp   but there is a gap in the bound~\eqref{ole1}.
}
\end{example}

\begin{remark}{\rm For a metric random walk space $[X, d, m]$ with invariant and reversible  probability measure~$\nu$, Y.~Ollivier in~\cite[Corollary~31]{O}, under the assumption that {\white\eqref{cond3int}}
\begin{equation}\label{cond3int}\int \int \int d(y,z)^2 dm_x(y) dm_x(z) d \nu(x) < +\infty,\end{equation}
proves that if the Ollivier-Ricci curvature $\kappa_m $ is positive and the space is ergodic, then
$[X,d,m,\nu]$ satisfies the Poincar\'{e} inequality \begin{equation}\label{OlPoinca1}
\kappa_m {\rm Var}_\nu(f) \leq \mathcal{H}_m(f) \quad \hbox{for all} \ f \in L^2(X,\nu),
\end{equation}
and, consequently,
$$\kappa_m \leq  {\rm gap}(-\Delta_m).$$
 Observe that, in fact, the ergodicity follows from the positivity of $\kappa_m$ (Theorem \ref{1despf01}).
}
\end{remark}

\subsection{ Transport Inequalities}

  Given a metric random walk  space $[X,d,m]$ we define, for $x \in X$,
$$\Theta(x):=  \frac12  \left(W^d_2(\delta_x, m_x)\right)^2 =\frac12 \int_X d(x,y)^2 dm_x(y) ,$$
and $$\Theta_m :=  \operatorname*{ess\,sup}_{x \in X}  \Theta(x) .$$
Since $\Theta(x) \leq \frac12\left(\hbox{diam}({\rm supp}(m_x) \right)^2,$ if $\hbox{diam}(X)$ is finite then we have
$\Theta_m \leq \frac12(\hbox{diam}(X))^2$.
Observe that
\begin{equation}\label{bb1observe}
\Vert \Gamma(f) \Vert_\infty=  \sup_{x \in X} \frac12 \int_X (f(x) - f(y))^2 dm_x(y) \leq \Theta_m \Vert f \Vert_{Lip}^2.
\end{equation}

Given a metric measure space $(X,d, \mu)$ as in Example \ref{JJ} (\ref{dom006}), if  $m^{\mu,\epsilon} $ is the $\epsilon$-step random walk associated to $\mu$, that is
 $$m^{\mu,\epsilon}_x:= \frac{\mu \res B(x, \epsilon)}{\mu(B(x, \epsilon))} \quad \hbox{for } x\in X,$$
 then
 $$\Theta_{m^{\mu,\epsilon}} \leq \frac12 \epsilon^2. $$

 Let $J_m(x)$ be the {\it jump} of the random walk at $x$:
$$J_m(x):= W^d_1(\delta_x, m_x) = \int_X d(x,y) dm_x(y).$$
In the particular case of the metric random walk space associated to a locally finite discrete graph $[V(G), d_G, m^G ]$, we have
$$
J_{m^G}(x) =  \frac{1}{d_x} \sum_{  y \sim x, \,  y \not= x} w_{xy} \leq 1,
$$
thus
\begin{equation}\label{grppg11}
\Theta(x) = \frac12 J_{m^G}(x) =  \frac{1}{2d_x} \sum_{x \sim y, \  x \not= y} w_{xy} \leq \frac12.
\end{equation}

In the case of the metric random walk space $[\Omega, \Vert  . \Vert, m^J]$ (see Example \ref{Jheat} (2)), with $J(x)= \frac{1}{\vert B_r(0)\vert}\1_{B_r(0)}$, a simple computation gives
$$ \Theta(x) \leq \frac{N}{2(N+2)}r^2.$$

It is well known that in the case of diffusion semigroups the Bakry-\'{E}mery curvature-dimension condition  $BE(K, \infty)$ of its generator is characterized by gradient estimates on the semigroup (see for instance \cite{Bakry} or \cite{BGL}). The same characterization is also true for  weighted discrete graphs (see for instance \cite{KKRT} and \cite{ChL}).  With a similar proof we have that in the general context of metric random walk spaces this characterization is also true.

\begin{theorem}\label{ok1}   Let $[X, d, m]$ be a metric random walk space with invariant-reversible  probability measure~$\nu$ and let $(T_t)_{t>0} = (e^{t \Delta_m})_{t>0}$ be the heat semigroup. Then, $\Delta_m$ satisfies the  Bakry-\'{E}mery curvature-dimension condition   $BE(K, \infty)$ with $K >0$ if, and only if,
\begin{equation}\label{E1ok1}
\Gamma(T_t f) \leq e^{-2Kt} T_t(\Gamma(f)) \quad \forall \, t \geq 0, \ \forall\, f \in L^2(X, \nu).
\end{equation}
\end{theorem}
\begin{proof}
  Fix $t > 0$. For $s \in [0,t)$, we define the function
$$g(s,x):=e^{-2Ks}T_s(\Gamma(T_{t-s}f))(x), \ \ x\in X.$$
The same computations as in \cite{KKRT} show that
$$\frac{\partial g}{\partial s}(s,x)=2e^{-2Ks}T_s\left(\Gamma_2(T_{t-s}f)-K\Gamma(T_{t-s}f)\right)(x).$$
Then, if $\Delta_m$ satisfies the  Bakry-\'{E}mery curvature-dimension condition   $BE(K, \infty)$ with $K >0$, we have that $\frac{\partial g}{\partial s}(s,x) \geq 0$ which is equivalent to \eqref{E1ok1}.
On the other hand, if \eqref{E1ok1} holds, we have $\frac{\partial g}{\partial s}(0,x) \geq 0$, which is equivalent to
$$\Gamma_2(T_{t}f)-K\Gamma(T_{t}f) \geq 0. $$ Then, letting $t \to 0$, we get $\Gamma_2(f)-K\Gamma(f) \geq 0.$
\end{proof}

The {\it Fisher-Donsker-Varadhan information} of a probability measure $\mu$ on $X$ with respect to $\nu$ is defined by
$$I_\nu(\mu):= \left\{\begin{array}{ll} 2\mathcal{H}_m (\sqrt{f}) \quad &\hbox{if} \ \mu = f\nu, \ f\ge 0,
 \\ \\ + \infty, \quad &\hbox{otherwise}. \end{array}  \right.$$
 Observe  that $$D(I_\nu)=\{\mu: \mu=f\nu,\ f\in L^1(X,\nu){}^+\}$$ since $\sqrt{f}\in L^2(X,\nu)=D(\mathcal{H}_m)$ whenever $f\in L^1(X,\nu){}^+$.
  Here, we use the notation $L^p(X,\nu)^+:=\{f\in L^p(X,\nu)\,:\, f\ge 0\ \, \nu-\hbox{a.e.}\}$.

In the next result we show that the  Bakry-\'{E}mery curvature-dimension condition   $BE(K, \infty)$ with $K >0$ implies a transport-information inequality, result that was obtained for the particular case of Markov chains in discrete spaces in \cite{FS}.

\begin{theorem}\label{transptinq1} Let $[X, d, m]$ be a metric random walk space with invariant-reversible  probability measure~$\nu$, and assume that  $\Theta_m$  is finite. If $\Delta_m$ satisfies the  Bakry-\'{E}mery curvature-dimension condition   $BE(K, \infty)$  with $K >0$, then   $\nu$ satisfies the transport-information inequality
 \begin{equation}\label{e1transptinq12}
 W_1^{d}(\mu, \nu) \leq \frac{\sqrt{\Theta_m}}{K} \sqrt{I_\nu(\mu)}, \quad \hbox{for all probability measures $\mu \ll \nu$.}
\end{equation}
\end{theorem}

\begin{proof} Let $\mu$ be a probability measure $\mu \ll \nu$, and set $\mu =f\nu$.
By the Kantorovich-Rubinstein Theorem we have that
$$W_1^{d}(\mu, \nu) = \sup \left\{ \int_X g(x) (f(x) - 1) d\nu(x) \ : \ \  \Vert g \Vert_{Lip} \leq 1  \  \hbox{and} \ \ g \ \hbox{bounded}  \right\}.$$

 Let $T_t = e^{t \Delta_m}$  be the heat semigroup. Given $g \in L^\infty(X, \nu)$ with $\Vert g \Vert_{Lip} \leq 1$, having in mind Proposition \ref{asinttot}, we have
 $$ \int_X g(x)(f(x) - 1) d\nu(x)  = - \int_0^{+\infty} \frac{d}{dt} \int_X (T_t g)(x) f(x) d\nu(x) dt = - \int_0^{+\infty}  \int_X \Delta_m (T_t g)(x) f(x) d\nu(x) dt $$ $$= \int_0^{+\infty} \mathcal{E}_m (T_t g, f) dt = \int_0^{+\infty}  \int_X \Gamma(T_t g, f)(x) d \nu(x) dt.$$
Now, using the Cauchy-Schwartz inequality, the reversibility  of the measure $\nu$ and that $$(\sqrt{f}(y) + \sqrt{f}(x))^2 \leq 2((f(x) + f(y)),$$ we have
$$\begin{array}{l}\displaystyle \int_X \Gamma(T_t g, f)(x) d \nu(x) = \frac12 \int_{X\times X} ((T_t g)(y) - (T_t g)(x)) (f(y) - f(x)) dm_x(y) d \nu(x) \\ \\ \displaystyle = \frac12 \int_{X\times X} ((T_t g)(y) - (T_t g)(x)) (\sqrt{f}(y) - \sqrt{f}(x)) (\sqrt{f}(y) + \sqrt{f}(x))dm_x(y) d \nu(x)\\ \\ \displaystyle \leq \left( \int_{X\times X} \frac14 (\sqrt{f}(y) - \sqrt{f}(x))^2 dm_x(y) d \nu(x) \right)^{\frac12}
\\ \\ \displaystyle \qquad \times \left( \int_{X\times X} ((T_t g)(y) - (T_t g)(x))^2 (\sqrt{f}(y) + \sqrt{f}(x))^2dm_x(y) d \nu(x)\right)^{\frac12}\\ \\ \displaystyle \leq  \left(\frac12 \int_X \Gamma( \sqrt{f}) (x) d \nu(x) \right)^{\frac12}  \left(4 \int_X \left( \int_X((T_t g)(y) - (T_t g)(x))^2  dm_x(y) \right) f(x) d\nu(x)\right)^{\frac12}. \end{array} $$
 Then, applying Theorem \ref{ok1}, we get
$$\begin{array}{c}\displaystyle \int_X g(x)(f(x) - 1) d\nu(x) \leq \left(\frac12  \mathcal{H}_m(\sqrt{f})\right)^{\frac12}  \int_0^{+\infty} \left(4 \int_X \Gamma((T_t g)(x))f(x) d\nu(x) \right)^{\frac12}dt
 \\ \\ \displaystyle
 \leq  \left(  2\mathcal{H}_m(\sqrt{f})\right)^{\frac12} \int_0^{+\infty} \left( e^{-2Kt} \int_X T_t\big(\Gamma(g)\big) (x) f(x) d\nu(x) \right)^{\frac12} dt.
 \end{array}$$
 Now, by \eqref{bb1observe} and  \eqref{CAcont}, we have
 $$ \vert  T_t\big(\Gamma(g)\big) (x) \vert \leq \Vert  T_t\big(\Gamma(g)\big) \Vert_\infty \leq \Vert \Gamma(g) \Vert_\infty \leq \Theta_m.$$
 Hence
 $$\displaystyle \int_X g(x)(f(x) - 1) d\nu(x) \leq \left( 2 \mathcal{H}_m(\sqrt{f})\right)^{\frac12} \int_0^{+\infty} \left( e^{-2Kt} \Theta_m  \int_X f(x) d\nu(x) \right)^{\frac12} dt \leq \frac{\sqrt{\Theta_m}}{K}\left( 2 \mathcal{H}_m(\sqrt{f})\right)^{\frac12}$$
Finally, taking the supremum over all bounden functions $g$  with $\Vert g \Vert_{Lip} \leq 1$ we get  \eqref{e1transptinq12}.
\end{proof}

\begin{remark}\label{faacil}{\rm If $\nu$ satisfies a transport-information inequality
\begin{equation}\label{OllerR}
W_1^d(\mu, \nu) \leq \lambda  \sqrt{2\mathcal{H}_m(\sqrt{f})}  \qquad \hbox{for  all $\mu =f\nu$,}
\end{equation}
then $\nu$ is ergodic. In fact, if $\nu$ is not ergodic,  then  by Theorem \ref{ERggB} there exists $D \subset X$ with $0 < \nu(D) < 1$ such that $\Delta_m \1_D = 0$. Now,  if $\mu:= \frac{1}{\nu(D)} \1_D \nu$, then $\mu  \not= \nu$ and, therefore, by \eqref{OllerR}, we get $\mathcal{H}_m(\1_D) >0$, which is a contradiction with $\Delta_m \1_D = 0$.
}
\end{remark}

 As a consequence of the previous Remark and Theorem \ref{transptinq1}, we have that the positivity of the  Bakry-\'{E}mery curvature-dimension condition implies ergodicity of $\Delta_m$, then, by Theorem \ref{madrid1}, we have the following result.

\begin{theorem}\label{Nmadrid1}  Let $[X, d, m]$ be a  metric random walk space with invariant-reversible  probability measure~$\nu $, and assume that  $\Theta_m$  is finite. Then:

 If  $\Delta_m$ satisfies  the  Bakry-\'{E}mery curvature-dimension condition $BE(K,n)$ with $K>0$, we have
$${\rm gap}(- \Delta_m) \geq K\frac{n}{n-1}.$$

In the case that $\Delta_m$ satisfies  the   Bakry-\'{E}mery curvature-dimension condition $BE(K,\infty)$ with $K >0$, we have
$${\rm gap}(- \Delta_m) \geq K.$$
\end{theorem}

The {\it relative entropy}  of $0 \leq \mu \in \mathcal{M}(X)$ with respect to $\nu$ is defined by
$${\rm Ent}_\nu(\mu):=\left\{\begin{array}{ll} \displaystyle \int_X f \log f d\nu - \nu(f) \log\big(\nu(f)\big)  & \hbox{if }\mu=f\nu,\ f\ge 0,\ f \log f \in L^1(X,\nu),\\[10pt] +\infty,&\hbox{otherwise,}\end{array}\right.$$ with the usual convention that $f(x) \log f(x) = 0$ if $f(x) = 0$.

 The next result   shows that a transport-information inequality implies a transport-entropy inequality and, therefore, normal concentration (see for example~\cite{BG,Ledoux}).

\begin{theorem}\label{transptinq11} Let $[X, d, m]$ be a metric random walk space with invariant-reversible  probability measure~$\nu$ and assume that  $\Theta_m$  is finite  and that there exists some $x_0 \in X$ such that $\int d(x,x_0) d\nu(x) < \infty$.  Then the transport-information inequality
\begin{equation}\label{Oller}
W_1^d(\mu, \nu) \leq \frac{1}{K}\sqrt{I_\nu(\mu)}, \quad \hbox{for all probability measures $\mu \ll \nu$,}
\end{equation}
implies  the transport-entropy inequality
\begin{equation}\label{ole}
 W_1^{d}(\mu, \nu) \leq  \sqrt{\frac{\sqrt{2\Theta_m}}{K} \, {\rm Ent}_\nu \left(\mu \right)}, \quad \hbox{for all probability measures $\mu \ll \nu$.}
\end{equation}
\end{theorem}

\begin{proof} By \cite[Theorem 1.3]{BG}, inequality \eqref{ole} is equivalent to   \begin{equation}\label{Nconcent1}
    \int_X e^{\lambda f(x)} d \nu(x) \leq e^{\lambda^2  \frac{\sqrt{\Theta_m}}{2\sqrt{2}K}},
    \end{equation}
for every bounded function $f$ on $X$ with $\Vert f \Vert_{Lip} \leq 1$ and $\nu(f) = 0$, and all $\lambda \in \R$.

Given  $f \in L^\infty(X,\nu)$  with $\Vert f \Vert_{Lip} \leq 1$ and $\nu(f) = 0$, we define the function $$\Lambda(\lambda):= \int_X e^{\lambda f(x)} d\nu(x),$$
and the probabilities
$$\mu_{\lambda}:= \frac{1}{\Lambda(\lambda)} e^{\lambda f} d\nu.$$
By the Kantorovich-Rubinstein Theorem and the assumption \eqref{Oller}, we have
$$\frac{d}{d\lambda} \log(\Lambda (\lambda)) = \frac{1}{\Lambda(\lambda)} \int_X f(x) e^{\lambda f(x)} d \nu(x) =  \int_X f(x)(d\mu_{\lambda}(x) - d \nu(x)) \leq W_1^{d}(\mu_{\lambda}, \nu)\le$$ $$\leq \frac{1}{K} \sqrt{2\mathcal{H}_m\left(\sqrt{\frac{1}{\Lambda(\lambda)} e^{\lambda f}}\right)} = \frac{\sqrt{2}}{K} \sqrt{\int_X \Gamma\left(\sqrt{\frac{1}{\Lambda(\lambda)} e^{\lambda f}}\right)(x) d\nu(x)} $$ $$ = \frac{\sqrt{2}}{K}\sqrt{\int_X \frac{1}{\Lambda(\lambda)} \Gamma\left( e^{\frac{\lambda f}{2}}\right)(x) d\nu(x) }.$$
Now, since $ 1 - \frac{1}{a}  \leq \log a$ for $a \geq 1$, having in mind the reversibility of $\nu$, we have
$$\int_X \Gamma(g)(x) d \nu(x) \leq \int_X g^2(x)\Gamma(\log g)(x) d \nu(x),$$
and, consequently, by \eqref{bb1observe}, we get
$$\frac{d}{dt} \log(\Lambda (\lambda)) \leq \frac{\sqrt{2}}{K}\sqrt{\int_X \frac{1}{\Lambda(\lambda)} e^{\lambda f(x)}\Gamma\left(\frac{\lambda f}{2}\right)(x) d \nu(x)} =\frac{\lambda }{\sqrt{2}K}\sqrt{\int_X \frac{1}{\Lambda(\lambda )} e^{\lambda f(x)}\Gamma\left(f\right)(x) d \nu(x)}$$ $$=\frac{\lambda }{\sqrt{2}K}\sqrt{\int_X \Gamma\left(f\right)(x) d \mu_{\lambda}(x)} \leq  \frac{\sqrt{\Theta_m}}{\sqrt{2}K}\lambda.$$
Then, integrating we get \eqref{Nconcent1}.

\end{proof}

In the next example we  see that, in general, a transport-entropy inequality does not imply transport-information inequality.

\begin{example}\label{bonitoejemplo1}{\rm Let  $\Omega= [-1,0] \cup [2,3]$ and consider the metric  random walk space $[\Omega,d,m^{J,\Omega}]$, with $d$ the Euclidean distance in $\R$ and $J(x) = \frac{1}{2} \1_{[-1,1]}$ (see Example~\eqref{JJ}~\eqref{dom00606}). An invariant and reversible probability measure for  $m^{J,\Omega}$ is $\nu:= \frac{1}{2} \mathcal{L}^1\res \Omega$. By the Gaussian integrability criterion \cite[Theorem 2.3]{DGW} $\nu$
satisfies a transport-entropy inequality. However, $\nu$ does not satisfy a transport-information inequality, since if $\nu$ satisfies a transport-information inequality, then $\nu$ must be ergodic (see Remark \ref{faacil}). Now it is easy to see that $[\Omega,d,m^{J,\Omega}]$ is not $m$-connected and then by Theorem \ref{ergconect}, $\nu$ is not ergodic.
}
\end{example}

By Theorems \ref{1despf01} and \ref{ergconect}, we have that the metric  random walk space $[\Omega,d,m^{J,\Omega}]$ of the above example  has non-positive Ollivier-Ricci curvature. In the next theorem we will see that, under positive Ollivier-Ricci curvature, a transport-information inequality holds. First we need the following result.

\begin{lemma}\label{megustalema} Let $[X, d, m]$ be a metric random walk space with invariant-reversible  probability measure~$\nu$. Then, if $f \in L^2(X,\nu)$ with $\Vert f \Vert_{Lip} \leq 1$, we have $\Vert e^{t\Delta_m} f \Vert_{Lip} \leq e^{-t \kappa_m }$.
\end{lemma}

\begin{proof} By \cite[Proposition 25]{O}, we have that
$$\kappa_{m^{\ast (n+l)}} \geq \kappa_{m^{\ast n}}+\kappa_{m^{\ast l}} - \kappa_{m^{\ast n}}\kappa_{m^{\ast l}} \quad \forall \, n,l \in \N.$$
where $\kappa_{m^{\ast 1}}  = \kappa_m$. Hence,
\begin{equation}\label{ole1repetido}1 - \kappa_{m^{\ast n}} \leq (1 - \kappa_{m})^n \quad \forall \, n \in \N.\end{equation}
By Theorem \ref{expannsion1} and \eqref{ole1repetido}, we have
$$
 \vert e^{t\Delta_{m}} f(x) - e^{t\Delta_{m}} f(y) \vert =  \left\vert  e^{-t}\sum_{n=0}^{+\infty}\int_{X} f(z)(dm_x^{\ast n}(z) - dm_y^{\ast n}(z))\frac{t^n}{n!} \right\vert $$ $$ \leq e^{-t} \sum_{n=0}^{+\infty} W^d_1(m_x^{\ast n},m_y^{\ast n})\frac{t^n}{n!} \le e^{-t} \sum_{n=0}^{+\infty} (1 - \kappa_{m^{\ast n}}) d(x,y)\frac{t^n}{n!} \leq e^{-t} \sum_{n=0}^{+\infty} (1 - \kappa_{m})^n\frac{t^n}{n!} d(x,y)$$ $$= e^{-t} e^{t(1 - \kappa_{m})}d(x,y)= e^{-t \kappa_{m}}d(x,y),$$
 from where it follows that
 $\Vert e^{t\Delta_m} f \Vert_{Lip} \leq e^{-t \kappa_m }$.
\end{proof}

\begin{theorem}\label{megustados}   Let $[X, d, m]$ be a metric random walk space with invariant-reversible  probability measure~$\nu$,   and assume that  $\Theta_m$  is finite. If $\kappa_m >0$    then the following transport-information inequality holds:
 \begin{equation}\label{perfectdos}W_1^{d}(\mu, \nu)  \leq \frac{\sqrt{2\Theta_m}}{\kappa_m} \sqrt{I_\nu(\mu)}, \quad \hbox{for all probability measures $\mu \ll \nu$. }
 \end{equation}
\end{theorem}

\begin{proof}
 Let $T_t = e^{t \Delta_m}$  be the heat semigroup and $\mu = f \nu$ be a probability measure in $X$. We use,  as in the proof of Theorem \ref{transptinq1}, the Kantorovich-Rubinstein Theorem. Let  $g \in L^\infty(X, \nu)$ with $\Vert g \Vert_{Lip} \leq 1$. Having in mind Lemma \ref{megustalema}, we have
 $$  \int_X g(x)(f(x) - 1) d\nu(x) = - \int_0^{+\infty} \frac{d}{dt} \int_X (T_t g)(x) f(x) d\nu(x) dt   =  - \int_0^{+\infty}  \int_X \Delta_m (T_t g)(x) f(x) d\nu(x) dt $$ $$ =  \int_0^{+\infty}  \frac{1}{2}\int_{X \times X} ((T_t g)(y) - (T_t g)(x)) (f(y) -f(x)) dm_x(y)d\nu(x) dt $$ $$ \leq   \int_0^{+\infty}  \Vert T_t g \Vert_{Lip} \frac{1}{2}\int_{X \times X} \ d(x,y)\vert f(y) -f(x) \vert dm_x(y)d\nu(x) dt   $$ $$\leq  \int_0^{+\infty}
 e^{-t \kappa_m } \frac12\int_{X \times X} d(x,y) \vert f(y) -f(x) \vert dm_x(y)d\nu(x) dt$$ $$=\frac{1}{2\kappa_m }
\int_{X \times X}  d(x,y) \vert f(y) -f(x) \vert dm_x(y)d\nu(x)  $$ $$ =\frac{1}{2\kappa_m }
\int_{X \times X}    d(x,y)\vert \sqrt{f}(y) - \sqrt{f}(x) \vert \left(  \sqrt{f}(y) +  \sqrt{f}(x)\right) dm_x(y)d\nu(x)$$ $$\leq \frac{\sqrt{2}}{2\kappa_m} \sqrt{\mathcal{H}_m(\sqrt{f})} \sqrt{\int_{X \times X}d^2(x,y)\left(  \sqrt{f}(y) +  \sqrt{f}(x)\right)^2dm_x(y)d\nu(x)}.$$
Now, using reversibility of $\nu$,
$$\int_{X \times X}d^2(x,y)\left(  \sqrt{f}(y) +  \sqrt{f}(x)\right)^2dm_x(y)d\nu(x) $$ $$= \int_{X \times X}d^2(x,y)\left(2 f(x) + 2 f(y) - \left(\sqrt{f}(y) - \sqrt{f}(x)\right)^2\right)dm_x(y)d\nu(x)$$ $$ \le 2\int_{X \times X}d^2(x,y)\left( f(x) +  f(y) \right)dm_x(y)d\nu(x)    \le 8\Theta_m.$$
Therefore, we get
 $$    \int_X g(x)(f(x) - 1) d\nu(x) \leq \frac{2 \sqrt{\Theta_m}}{\kappa_m} \sqrt{\mathcal{H}_m(\sqrt{f})} , $$
 so, taking the supremum over the functions $g$,
 \[  W_1^{d}(\mu, \nu)  \leq   \frac{  \sqrt{2\Theta_m}}{\kappa_m} \sqrt{2\mathcal{H}_m(\sqrt{f})}= \frac{\sqrt{2\Theta_m}}{\kappa_m} \sqrt{I_\nu(\mu)}. \qedhere\]
\end{proof}

\noindent {\bf Acknowledgment.} The authors have been partially supported  by the Spanish MCIU and FEDER, project PGC2018-094775-B-I00.  The second author was
also supported by Ministerio de Econom\'{\i}a y Competitividad under Grant BES-2016-079019.

\end{document}